\documentclass{preprint}

\Title{Analysis and Approximation of\\ Stochastic Nerve Axon Equations}
\ShortTitle{Stochastic Nerve Axon Equations}
\Author{Martin Sauer and Wilhelm Stannat} 
\Address{Institut f\"ur Mathematik, Technische Universit\"at Berlin, Stra\ss e des 17. Juni 136, D-10623 Berlin, Germany and Bernstein Center for Computational Neuroscience, Philippstr. 13, D-10115 Berlin, Germany. E-mail address: \texttt{sauer@math.tu-berlin.de}, \texttt{stannat@math.tu-berlin.de}.}
\Abstract{We consider spatially extended conductance based neuronal models with noise described by a stochastic reaction diffusion equation with additive noise coupled to a control variable with multiplicative noise but no diffusion. We only assume a local Lipschitz condition on the nonlinearities together with a certain physiologically reasonable monotonicity to derive crucial $L^\infty$-bounds for the solution. These play an essential role in both the proof of existence and uniqueness of solutions as well as the error analysis of the finite difference approximation in space. We derive explicit error estimates, in particular a pathwise convergence rate of $\sqrt{\sfrac{1}{n}}-$ and a strong convergence rate of $\sfrac1n$ in special cases. As applications, the Hodgkin-Huxley and FitzHugh-Nagumo systems with noise are considered.}
\AMSsub{Primary 60H15, 60H35, Secondary 35R60, 65C30, 92C20}

\newcommand{\Dom}{\mathcal{O}}

\newcommand{\bfx}{\mathbf{x}}
\newcommand{\bfy}{\mathbf{y}}
\newcommand{\bfX}{\mathbf{X}}
\newcommand{\dwit}{\,\mathrm{d}W_i(t)}

\newcommand{\tu}{\tilde{u}}
\newcommand{\tU}{\tilde{U}}
\newcommand{\tbfX}{\tilde{\bfX}}

\newcommand{\intwithlimits}[1]{\int_{\frac{#1-1}{n}}^{\frac{#1}{n}}}
\newcommand{\sumk}{\sum_{k=1}^n}

\newcommand{\sumkl}{{\kern -0.5em}\sum_{k=1,l=0}^n {\kern -0.5em}}
\newcommand{\vek}[2]{\begin{pmatrix}#1 \\ #2 \end{pmatrix}}

\begin{document}
\section{Introduction}
In the 1950ies, Hodgkin and Huxley \cite{HodgkinHuxley} derived a system of nonlinear equations describing the dynamics of a single neuron in terms of the membrane potential, experimentally verified in the squid's giant axon. In particular, they found a model for the propagation of an action potential along the axon, which is essentially the basis for all subsequent conductance based models for active nerve cells. The system consists of one partial differential equation for the membrane potential $U$
\begin{equation}\label{eq:HHU}
\tau \partial_t U = \lambda^2 \partial_{xx} U - g_{\text{K}} (\bfX) (U - E_{\text{K}}) - g_{\text{Na}} (\bfX) (U - E_{\text{Na}}) - g_{\text{L}} (U - E_{\text{L}})
\end{equation}
and three ordinary differential equations for the gating variables $\bfX = (n, m, h)$ -- describing the probability of certain ion channels being open -- given by
\begin{equation}\label{eq:HHX}
\partial_t \bfX = \mathbf{a} (U) (1 - \bfX)  - \mathbf{b}(U) \bfX.
\end{equation}
For a more detailed description we refer the reader to \cite{Ermentrout} or Section \ref{sec:HH}, where we also introduce all coefficients and constants used above. For well-posedness of these equations we refer to \cite{Mascagni}.

In this article, we study equations of such type under random fluctuations. For the ``The What and Where of Adding Channel Noise to the Hodgkin-Huxley Equations'' we refer to \cite{GoldwynSheaBrown}. At this point, let us safely assume that the noise under consideration is justified. This procedure eventually leads to a stochastic partial differential equation, in particular a stochastic reaction diffusion equation for the variable $U$ coupled to a possibly vector valued auxiliary variable $\bfX$. The main mathematical challenge with such equations is that their coefficients do neither satisfy a global Lipschitz condition nor the standard monotonicity and coercivity conditions. Thus the results of this article are twofold, concerning both well-posedness and numerical approximation.

For the question of existence and uniqueness of solutions the standard methods from \cite{dpzSPDE} or \cite{prevot}, \cite{LiuLocallyMonotone}, \cite{LiuGeneralCoercive} do not apply. However, in the uncoupled system -- with fixed $\bfX$ and $U$ variable, respectively -- monotonicity is restored for each equation individually. This allows us to extend the existing results on variational solutions to cover such stochastic nerve axon equations via a fixed point iteration. Here, the key ingredient is a certain $L^\infty$-bound for the membrane potential $U$.

Concerning the numerical approximation we consider the well-known finite difference method for the spatial variable only. This method has been studied intensively, e.\,g. by Gy\"ongy \cite{Gyongy99}, Shardlow \cite{Shardlow}, Pettersson \& Signahl \cite{Pettersson05} and more recently by the authors \cite{SauerLattice} applied to the simpler FitzHugh-Nagumo equations. Although the method is heavily used in applied sciences, up to the best of our knowledge none of the existing literature covers a convergence result with explicit rates for such non-globally Lipschitz and non-monotone coefficients. For similar results using different approximation schemes we refer among others to \cite{GyongyMillet05}, \cite{GyongyMillet09} for some abstract approximation schemes and strong convergence rates, \cite{Hausenblas}, \cite{LordRougement}, \cite{GinzburgLandau}, \cite{Jentzen} for spectral Galerkin methods also for nonlinearities without a global Lipschitz condition, \cite{BloemkerBurgers} for Galerkin methods with non-diagonal covariance operator, \cite{Kruse} for optimal error estimates of a finite element discretization with Lipschitz continuous nonlinearities, \cite{AnderssonLarsson} for weak convergence of such a finite element method and finally a more or less recent overview concerning the numerical approximation of SPDE in \cite{KloedenJentzen}. Our proofs are based on It\^o's formula for the variational solution, the $L^\infty$-bound for the membrane potential and some uniform improved regularity estimates of the approximated solution. We deduce explicit error estimates, which a priori do not yield a strong convergence rate but only pathwise convergence with smaller rate $\sfrac12-$. In special cases, e.\,g. when the drift satisfies a one-sided Lipschitz condition as in \cite{SauerLattice}, one can improve the result to obtain a strong convergence rate of $\sfrac1n$.

The article is structured as follows. In the next section we describe the precise mathematical setting and all assumption on the coefficients. Section \ref{sec:Wellposed} is then devoted to the existence and uniqueness Theorem \ref{thm:Ex+Unique}, while Section \ref{sec:Approx} introduces the approximation scheme as well as states Theorems \ref{thm:Approx} and \ref{thm:Monotone} on convergence and explicit rates. We finish with two examples, on the one hand the Hodgkin-Huxley system mentioned before and also the FitzHugh-Nagumo equations studied in \cite{SauerLattice}. In particular, we are able to generalize and improve the results obtained there. The appendix contains more or less well-known facts about the stochastic convolution presented here for the reader's convenience.

\section{Mathematical Setting and Assumptions}
Let us first fix some notation. By $c$ and $C$ we denote constants which may change from line to line, the letter $K$ is reserved for a numerical constant that may be explicitly calculated. Let $\Dom = (0,1)$ and define $H \df L^2(\Dom)$ as well as $H_d \df \prod_{i=1}^d H$ and the same notation for other product spaces, too.

In this work we consider the following stochastic reaction-diffusion equation on $H$ coupled nonlinearly to a system of $ d \geq 1$ equations on $H_d$ without diffusion.
\begin{equation}\label{eq:SRDE}
\begin{split}
\mathrm{d} U(t) &= \Big( A U(t) + f\big(U(t), \bfX(t)\big) \Big) \dt + B \dwt\\
\mathrm{d} X_i(t) &= f_i \big(U(t), X_i(t)\big) \dt + B_i \big( U(t), \bfX(t)\big) \dwit, \quad 1 \leq i \leq d\\
\end{split}
\end{equation}
subject to initial conditions $U(0) = u_0$, $\bfX(0) = \bfx_0$ and driven by $d+1$ independent cylindrical Wiener processes $W$, $W_i$ on $H$ with underlying complete, filtered probability space $(\Omega, \algF, \algF_t, \PP)$ and coefficients to be specified below. Note that we use bold symbols for $\bfx \in \R^d$ vector fields with components $x_i$ to discriminate between them and the scalar $U$ variable.

For the linear part of the drift in \eqref{eq:SRDE} we assume $(A u)(x) \df \partial_{xx} u(x)$ equipped with homogeneous Neumann boundary conditions $\partial_x u(0) = 0$ and $\partial_x u(1) = 0$, hence a linear operator $(A, D(A))$ on $H$. It is well known that $A$ is non-negative and self-adjoint with corresponding closed, symmetric form $\mathcal{E}(u,v) = - \int \partial_x u \partial_x v \dx$, $D(\mathcal{E}) = W^{1,2}(\Dom) \fd V$ and thus can be uniquely extended to an operator $A : V \to V^\ast$. Here, $V^\ast$ denotes the dual space of $V$. In order to study \eqref{eq:SRDE} in the framework of variational solutions we introduce the Gelfand triple $V \hookrightarrow H \hookrightarrow V^\ast$ with continuous and dense embeddings. Denote by $\dualp{\cdot}{\cdot}$ the dualization between $V$ and $V^\ast$, then it follows that $\dualp{u}{v} = \scp{u}{v}_H$ for all $u \in H$, $v \in V$.
\begin{rem}
From the modeling perspective the homogeneous Neumann boundary conditions are called \emph{sealed ends}, meaning no currents can pass the boundary. It is reasonable to assume that an input signal is received via an \emph{injected current} at one end (in our case $x=0$) thus we should replace $\partial_x u(0) = I(t)$, where $I \in C_b^\infty([0,T])$ is the input signal. However, it is standard to transform such a problem to a homogeneous boundary with modified right hand side. In particular, the drift $f$ then depends on time $t$ and space variable $x$. Under the assumptions on $I$ above, this does not modify the essential parts of the analysis and we neglect this for the sake of a concise presentation.
\end{rem}

The reaction part of the drift should satisfy the following conditions.
\begin{assum}\label{assum:Drift}
Let $f, f_i \in C^1(\R \times \R^d; \R)$ with
\begin{align*}
\abs{f(u,\bfx)}, \abs{\nabla f(u,\bfx)} &\leq L \big( 1 + \abs{u}^{r - 1} \big) \big( 1 + \rho(\bfx) \big), &\partial_u f(u, \bfx) &\leq L \big( 1 + \rho(\bfx)\big)\\
\abs{f(u, x_i)}, \abs{\nabla f_i(u,x_i)} &\leq L \big( 1 + \rho_i(u) \big) \big( 1 + \abs{x_i}\big), &\partial_{x_i} f_i(u, x_i) &\leq L
\end{align*}
for constants $L> 0$, $2 \leq r \leq 4$, some locally bounded functions $\rho: \R^d \to \R^+$, $\rho_i:\R \to \R^+$ and all $u \in \R, \bfx \in \R^d$. Concerning the growth we only assume that there exists a constant $\alpha > 0$ such that $\rho_i(u) \leq \mathrm{e}^{\alpha \abs{u}}$ for all $u \in \R$.
\end{assum}
Formulated in words, we assume that all functions are locally Lipschitz continuous and do not prescribe any a priori control on the constants. Furthermore, $f$ as a function of $u$ and $f_i$ as a function of $x_i$  with all other variables fixed satisfies a one-sided Lipschitz condition. In order to deal with the growth of the Lipschitz constants in terms $\rho$ and $\rho_i$, our analysis is based on $L^\infty$-estimates for both variables $U$ and $\bfX$ and we therefore have to impose the following additional monotonicity condition.
\begin{assum}\label{assum:Invariance}
There exist $K \geq 0$, $\kappa_K > 0$ such that $\partial_u f(u,x) \leq - \kappa_K$ for all $\abs{u} > K$ and $\bfx \in [0,1]^d$. Furthermore $f_i(u, x_i) \geq 0$ if $x_i \leq 0$ and $f_i(u,x_i) \leq 0$ if $x_i \geq 1$ for all $u \in \R$.
\end{assum}
From a physical point of view, the second assumption corresponds to the invariance of $[0,1]^d$ for $\bfX$, which is natural since it represents some proportion or density. Concerning the noise in equation \eqref{eq:SRDE} we choose additive noise in $U$ and allow for multiplicative noise in $\bfX$ which then has to respect the natural bounds $0$ and $1$, see e.\,g. \cite[Section 2.1]{Faugeras} for a reasonable choice. The precise assumptions are stated below.
\begin{assum}\label{assum:Kernel}
Let $B \in L_2(H,V)$, in particular it admits an integral kernel of the form
\[
(B u)(x) = \int_0^1 b(x,y) u(y) \dy, \quad x \in \Dom, b \in W^{1,2}\big(\Dom^2\big).
\]
Also, let $B_i: H \times H_d \to L_2(H)$ with integral kernels
\[
\big(B_i(u, \bfx)v\big)(x) = \mathbbm{1}_{\{0 \leq \bfx \leq 1\}} \int_0^1 b_i\big(u(x), \bfx(x), x, y) v(y) \dy, \quad x \in \Dom, b_i(u, \bfx) \in L^2\big(\Dom^2\big)
\]
being Lipschitz continuous in the first two variables, i.\,e.
\[
\abs{b_i(u, \bfx, x, y) - b_i(v, \mathbf{y}, x, y)} \leq L \big( \abs{u - v} + \abs{\bfx - \mathbf{y}} \big)
\]
for all $x, y \in \Dom$, $u, v \in \R$ and $\bfx, \mathbf{y} \in \R^d$. Furthermore, assume that $B_i : V \times V_d \to L_2(H,V)$, in particular $b_i(u, \bfx) \in W^{1,2}(\Dom^2)$ for $u \in V, \bfx \in V_d$ with
\[
\norm{B_i(u, \bfx)}_{L_2(H,V)} = \norm{b_i(u, \bfx)}_{W^{1,2}(\Dom^2)} \leq L \big( \norm{u}_V + \norm{\bfx}_{V_d}\big).
\]
\end{assum}
\section{Existence and Uniqueness}\label{sec:Wellposed}
As mentioned before, the existence and uniqueness result is based on $L^\infty$-bounds for both variables, essentially based on the observation that $[0,1]^d$ is forward invariant for the dynamics of $\bfX$, as it is easily seen in e.\,g. the Hodgkin-Huxley equations. For this purpose define the set
\[
\mathcal{X} \df \Big\{ \bfX \in C\big([0,T]; H_d\big) ~\algF_t\text{-adapted}: 0 \leq \bfX(t)\leq 1 ~\PP-\text{a.\,s. for a.\,e. } x \text{ and all } t \in [0,T].\Big\},
\]
i.\,e. what should be the a priori solution set for the auxiliary variables. We will show later on, that given an initial value $\bfx_0$ in between these bounds it indeed holds that $\bfX \in \mathcal{X}$. Moreover, let us introduce the Ornstein-Uhlenbeck process $Y$ as the solution to
\[
\mathrm{d}Y(t) = A Y(t) \dt + B \dwt, \quad Y(0) = 0.
\]
Let $E \df C(\overline{\Dom}; \R)$. The statistics of $Y$ are well known, in particular by Lemma \ref{app:Lem1} it follows that 
\begin{equation}\label{eq:RYt}
R_t^Y \df \sup_{s \in [0,t]} \norm{Y(s}_E < \infty 
\quad \PP\text{-a.\,s.}
\end{equation}
for all $t \in [0,T]$ is a Gaussian random variable. This motivates the following definition as a natural solution set for the $U$ variable.
\[
\mathcal{U} \df \left. \begin{cases} U \in  \algF_t\text{-adapted}: \norm{U(t)}_{L^\infty(\Dom)} \leq R_t ~\PP\text{-a.\,s. for all } t \in [0,T] \text{ for some $\algF_t$-adapted}\\\text{process $R_t$ with Gaussian moments, i.\,e. $\EV{\exp(\frac{\alpha}{2} R_T^2)}< \infty$ for some $\alpha > 0$.}\end{cases} {\kern -1em} \right\}.
\]
Due to the additive noise a process $U$ as a part of a solution to \eqref{eq:SRDE} cannot be uniformly bounded, however such a pathwise estimate is reasonable as we will show below. With this preliminary work we are able to state the following theorem.
\begin{thm}\label{thm:Ex+Unique}
Let $p \geq \max \{2(r-1), 4\}$, $u_0 \in L^p(\Omega, \algF_0, \PP; H)$ be independent of $W$ with Gaussian moments in $E$ and $\bfx_0 \in L^p( \Omega, \algF_0, \PP; H_d)$ with $0 \leq \bfx_0 \leq 1$ $\PP$-a.\,s. for a.\,e. $x$. Then there exists a unique variational solution $(U, \bfX)$ to \eqref{eq:SRDE} with $U \in \mathcal{U}$ and $\bfX \in \mathcal{X}$.
\end{thm}
The proof of this theorem is based on a fixed point iteration and solving each equation without coupling. This is carried out in the next three subsections.
\subsection{Solving the Equation for $U$}
Let us fix $\bfX \in \mathcal{X}$, then
\begin{equation}\label{eq:SPDEforU}
\mathrm{d}U(t) = \Big( AU(t) + f\big(U(t), \bfX(t)\big)\Big)\dt + B \dwt, \quad U(0) = u_0.
\end{equation}
\begin{lem}\label{lem:ExforU}
Let $u_0 \in L^p(\Omega, \algF_0, \PP; H)$, $p\geq 2(r-1)$. Given $\bfX \in \mathcal{X}$ there exists a unique variational solution $U$ satisfying
\[
\EV{ \sup_{t \in [0,T]} \norm{U(t)}_H^p + \int_0^T \norm{U(t)}_V^2 \dt} < \infty.
\]
\end{lem}
\begin{proof}
This lemma is a more or less immediate application of \cite[Theorem 1.1]{LiuLocallyMonotone}. As $\bfX \in \mathcal{X}$ it follows that $\rho(\bfX(t,x)) \leq \rho_0$ for some uniform (in $t$ and $x$) constant $\rho_0 > 0$. It remains to check the monotonicity and coercivity conditions (H1)--(H4) in \cite{LiuLocallyMonotone}. Of course, for $u, v, w \in V$ the map $s \mapsto \dualp{A(u+sv) + f(\bfX(t), u+sv)}{w}$ is continuous in $\R$. Monotonicity is also quite obvious, since we have a one-sided Lipschitz condition that implies
\[
2\dualp{A(u-v) + f\big((u, \bfX(t)\big) - f\big(v, \bfX(t)\big)}{u-v} \leq - 2\norm{u-v}_V^2 + 2L(1 + \rho_0) \norm{u-v}_H^2.
\]
This directly yields coercivity with the choice $v = 0$ and $\abs{f(0,\bfX(t))} \leq L(1 + \rho_0)$,
\[
2\dualp{Au + f\big( u, \bfX(t)\big)}{u} \leq - 2 \norm{u}_V^2 + 3L(1 + \rho_0) \norm{u}_H^2 + L (1 + \rho_0).
\]
The growth condition is based on the polynomial growth of $f$ of order $r-1 \leq 3$ and the Sobolev embedding $V \hookrightarrow L^\infty(\Dom)$ in dimension one, in detail
\begin{align*}
\norm{Au + f\big(u, \bfX(t)\big)}_{V^\ast} &\leq \norm{u}_V + \sup_{\norm{\phi}_V=1} \int_0^1 \abs{f \big(u, \bfX(t)\big)} \abs{\phi} \dx\\
&\leq \norm{u}_V + \sup_{\norm{\phi}_V=1} L ( 1 + \rho_0) \norm{\phi}_{L^\infty(\Dom)} \int_0^1 \big(1 + \abs{u}^{r-1}\big) \dx\\
&\leq C \big(1 + \norm{u}_V\big) \big(1 + \norm{u}_H^{r-2}\big).\qedhere
\end{align*}
\end{proof}
\begin{rem}
When there is no auxiliary variable $\bfX$ we can obtain strong solutions to \eqref{eq:SRDE} by \cite{GessGradient} without the upper bound $r \leq 4$. This approach makes use of the fact that the drift can be written as the gradient of a quasi-convex potential and it also derives more regularity, in particular $U \in L^2(\PP; L^2([0,T]; W^{2,2}(\Dom)))$. We would rather use such result instead, however in the present case the drift is time-dependent and it is unclear how and if the results of \cite{GessGradient} generalize. Instead of using such an improved regularity for the solution itself, the proof of the approximations results, e.\,g. the convergence rate obtained in Theorem \ref{thm:Approx} is based on more than the canonical regularity for the approximate solution, see Lemma \ref{lem:ImprovedAPriori}.
\end{rem}
Instead of proving the existence of more regular solutions and using Sobolev embedding to deduce $L^\infty$-estimates, we follow a different strategy that yields a pathwise bound and moreover is dimension independent.
\begin{lem}\label{lem:PathwiseBoundforU}
Let $U$ be the solution to \eqref{eq:SPDEforU} from Lemma \ref{lem:ExforU} and assume $u_0$ independent of $W$ having Gaussian moments in $E$. Then there exists an $\algF_t$-adapted stochastic process $R_t$ such that
\[
R_t \df \norm{u_0}_E + R + 2 R_t^Y \quad \PP\text{-a.\,s.}
\]
for some constant $R > 0$ and the process $R_t^Y$ specified in \eqref{eq:RYt} as the supremum of the corresponding Ornstein-Uhlenbeck process and the solution $U$ remains bounded by
\[
\norm{U(t)}_{L^\infty(\Dom)} \leq R_t, \quad \PP\text{-a.\,s.}
\]
for all $t \in [0,T]$. Moreover, $R_t$  has Gaussian moments, thus in particular $U \in \mathcal{U}$.
\end{lem}
\begin{proof}
At first, define $Z \df U - Y$, hence this difference satisfies a deterministic evolution equation with a random parameter
\begin{equation}
\ddt Z(t) = A Z(t) + f\big(Z(t) + Y(t), \bfX(t)\big), \quad Z(0) = u_0.
\end{equation}
Let $R > 0$ be some (possibly) large constant, essentially dependent on the shape of $f$. Then, define for fixed $t \in [0,T]$ with abuse of notation another process $R_t \df \norm{u_0}_{L^\infty} + R + R_t^Y$, of course with implicit $\omega$ dependence. We show that $U(s)$ does not leave the desired interval up to time $s \leq t$ using a cutoff via a normal contraction. To this end introduce for all $\eps > 0$  a normal contraction $\phi_\eps$ with $\phi_\eps \in C^\infty(\R)$, $\phi_\eps \leq \eps$, $\phi_\eps(u) = u + R_t$ for $u \leq - R_t$, $\abs{\phi_\eps'} \leq 1$ and $\abs{\phi_\eps''}\leq 2/\eps$. As $\eps \to 0$, $\phi_\eps$ approximates $\phi(u) \df \min \{0, u + R_t\}$. Obviously, $\phi_\eps(u) \in H$ (or $V$) if $u \in H$ (or $V$) by the contraction property. Thus we can calculate for $s \in [0,t]$
\begin{align*}
\frac{\mathrm{d}}{\mathrm{d}s} \norm{\phi_\eps\big(Z(s)\big)}_H^2 &= 2 \dualp{A Z(s) + f\big(Z(s) + Y(s), \bfX(s)\big)}{\phi_\eps'\big(Z(s)\big) \phi_\eps\big(Z(s)\big)}\\
&\leq - 2\int_0^1 \abs{\partial_x Z(s)}^2 \Big(\abs{\phi_\eps'\big(Z(s)\big)}^2 + \phi_\eps''\big(Z(s)\big) \phi_\eps\big(Z(s)\big) \Big) \dx\\
&\quad + 2 \scp{f\big(Z(s) + Y(s), \bfX(s)\big)}{\phi_\eps'\big(Z(s)\big) \phi_\eps\big(Z(s)\big)}_H.
\end{align*}
Concerning the first summand we know that $\abs{\phi_\eps'} \leq 1$, hence this term is finite and negative. Also, $\phi_\eps'' \phi_\eps \to 0$ point-wise as $\eps \to 0$ and $\abs{\phi_\eps'' \phi} \leq 2$. For the nonlinear part it holds that $\phi_\eps' \phi_\eps \to \phi$ point-wise as $\eps \to 0$ and $\abs{(\phi_\eps' \phi_\eps)(x)} \leq \abs{\phi_\eps(x)} \leq \abs{x} + R_t(\omega)$. We can integrate the inequality from $0$ up to $t$ and by Lebesgue's dominated convergence theorem we can interchange all integrals and the limit $\eps \to 0$ to obtain
\begin{align*}
\norm{\phi\big(Z(t)\big)}_H^2 &\leq \norm{\phi(u_0)}_H^2 + \int_0^t \scp{f \big(Z(s) + Y(s), \bfX(s)\big)}{\phi\big(Z(s)\big)}_H\\
&= \norm{\phi(u_0)}_H^2 + \int_0^t \scp{f\big(Y(s) -R_t, \bfX(s)\big)}{\phi \big(Z(s)\big)}_H \ds \\
&\quad + \int_0^t \scp{f\big(Z(s) + Y(s), \bfX(s)\big) - f\big( Y(s) - R_t, \bfX(s)\big)}{\phi \big(Z(s)\big)}_H \ds
\end{align*}
Now, the monotonicity in Assumption \ref{assum:Invariance} on $f$ implies that both of the integrals are less or equal to zero. In detail for $R > K > 0$ large enough the function $f(\cdot, \bfX(s))$ is monotone decreasing on $\R \setminus [-R, R]$ and in particular it changes its sign from $+$ to $-$. In both summands the integrand is zero if $Z(s) \geq - R_t$ because $\phi$ vanishes. In the opposite case
\[
Z(s) + Y(s) \leq Y(s) - R_t \leq - R
\]
and the integrand in the second integral is of the form
\[
\mathbbm{1}_{\{Z(s) \leq -R_t\}} \Big(f\big(Z(s) + Y(s), \bfX(s)\big) - f\big( Y(s) - R_t, \bfX(s)\big)\Big) \big( Z(s) + R_t\big) \leq 0.
\]
In the first integral we only need that $f(Y(s) - R_t, \bfX(s)) \geq 0$ since $\phi(Z(s)) \leq 0$. In conclusion, we have shown that
\[
\norm{\phi(Z(t))}_H \leq \norm{\phi(u_0)}_H \quad \Rightarrow \quad \einf_{x \in \Dom} Z(t) \geq -R_t \quad \PP\text{-a.\,s.}
\]
The corresponding upper bound can be obtained in the exact same way with $\tilde{\phi}(u) \df \max \{0, u - R_t\}$. This concludes the proof via the final estimate
\[
\norm{U(t)}_{L^\infty(\Dom)} \leq \norm{Z(t)}_{L^\infty(\Dom)} + \norm{Y(t)}_E \leq R_t + R_t^Y.
\]
Thus $U(t)$ is $\PP$-a.\,s. bounded by $R_t \df \norm{u_0}_E + R + 2 R_t^Y$ and the integrability follows from Lemma \ref{app:Lem1}.
\end{proof}
\subsection{Solving the Equation for $\bfX$}
Let us now fix $U \in \mathcal{U}$, then
\begin{equation}\label{eq:SPDEforX}
\mathrm{d}X_i(t) = f_i\big(U(t), X_i(t)\big) \dt + B_i \big(U(t), \bfX(t)\big) \dwit, \quad 1 \leq i \leq d
\end{equation}
with initial condition $\bfX(0) = \bfx_0$. For a compact notation introduce the vector fields $\mathbf{f}(t,\bfx) \df (f_i(U(t), x_i))_{1\leq i \leq d}$ and $\mathbf{B}(t,\bfx) \df (B_i(U(t), \bfx))_{1\leq i \leq d}$. 
\begin{lem}\label{lem:ExforX}
Let $p \geq 4$ and $\bfx_0 \in L^p(\Omega, \algF_0, \PP; H_d)$. Given $U \in \mathcal{U}$ there exists a unique strong solution $\bfX$ to \eqref{eq:SPDEforX} satisfying
\begin{equation}
\EV{ \sup_{t \in [0,T]} \norm{\bfX(t)}_{H_d}^p} < \infty.
\end{equation}
\end{lem}
\begin{proof}
This lemma is again an application of \cite[Theorem 1.1]{LiuLocallyMonotone}. We need to verify (H1)--(H4), this time in the Gelfand triple $H_d \hookrightarrow H_d \hookrightarrow H_d$ and again the hemicontinuity is straightforward to obtain since everything is a composition of continuous mappings. Also the monotonicity follows from the one-sided Lipschitz condition for each $f_i$ and the global Lipschitz assumption on $b_i$,
\[
\scp{ \mathbf{f}(t,\bfx) - \mathbf{f}(t,\bfy)}{\bfx - \bfy}_{H_d} + \norm{B(t,\bfx) - B(t,\bfy)}_{L_2(H_d)}^2 \leq \big(L + 2L^2\big) \norm{\bfx - \bfy}_{H_d}^2.
\]
Concerning coercivity we see that with the upper bound for $\abs{f_i(u,x_i)}$
\begin{align*}
\scp{\mathbf{f}(t, \bfx)}{\bfx}_{H_d} &= L \norm{\bfx}_{H_d}^2 + \scp{\mathbf{f}(t, 0)}{\bfx}_{H_d} \leq L \norm{\bfx}_{H_d}^2 + L \Big(\sum_{i=1}^d \big( 1 + \rho_i (R_t)\big)^2\Big)^\frac12 \norm{\bfx}_{H_d}\\
&\leq (L + 1) \norm{\bfx}_{H_d}^2 + \tfrac14 L^2 \sum_{i=1}^d \big( 1 + \rho_i (R_t)\big)^2 \fd \big(L + 1) \norm{\bfx}_{H_d}^2 + g_t.
\end{align*}
The stochastic process $g_t$ is $\algF_t$-adapted and in $L^p([0,T] \times \Omega, \dt \otimes \PP)$ for every $1 \leq p < \infty$ because each $\rho_i$ is of at most exponential growth. Also, the linear growth condition on $\mathbf{B}$ is immediate by the assumptions on $b_i$, since
\begin{align*}
\norm{\mathbf{B}(t,\bfx)}_{L_2(H_d)}^2 &\leq 2 \sum_{i=1}^d \Big( \norm{B_i\big( 0, 0 \big)}_{L_2(H)}^2 + \norm{B_i\big(U(t),\bfx\big) - B_i(0,0)}_{L_2(H)}^2\Big)\\
&\leq C + 4 L^2 \big( \norm{U(t)}_H^2 + \norm{\bfx}_{H_d}^2\big) \leq \tilde{g}_t + 4 L^2 \norm{\bfx}_{H_d}^2,
\end{align*}
where the stochastic process $\tilde{g}_t$ is again $\algF_t$-adapted and in every $L^p$. In a similar manner it follows that
\[
\norm{\mathbf{f}(t, \bfx)}_{H_d} = \Big( \sum_{i=1}^d \int_0^1 \abs{f_i \big(U(t), x_i\big)}^2 \dx \Big)^\frac12 \leq K g_t^\frac12 \big( 1 + \norm{\bfx}_{H_d}\big),
\]
hence the growth condition holds with $\beta =2$ and this guarantees the existence of a unique variational solution $\bfX$, which is indeed a strong solution since everything is $H_d$-valued. In particular, $\bfX$ satisfies
\begin{equation}\label{eq:IntegralEquationForX}
\bfX(t) = \bfx_0 + \int_0^t \mathbf{f}\big(s, \bfX(s)\big) \ds + \int_0^t \mathbf{B}\big(s, \bfX(s)\big) \,\mathrm{d} \mathbf{W}(s), \quad t \in [0,T]
\end{equation}
$\PP$-a.\,s. where $\mathbf{W} \df (W_i)_{1 \leq i \leq d}$.
\end{proof}
\begin{lem}\label{lem:InvarianceX}
Let $\bfX$ be the strong solution to \eqref{eq:IntegralEquationForX}. Assume $0 \leq \bfx_0 \leq 1$ $\PP$-a.\,s. for a.\,e. $x$, then $0 \leq \bfX(t) \leq 1$ $\PP$-a.\,s for a.\,e. $x$ and all $t \in [0,T]$.
\end{lem}
\begin{proof}
The proof is similar to the one of Lemma \ref{lem:PathwiseBoundforU} and involves the functions $\phi_0 (x) \df \min\{0, x\}$ and $\phi_1(x) \df \max \{1, x\}$. For $\eps > 0$ denote by $\phi_{j,\eps}$ the smooth normal contractions approximating $\phi_j$, $j=0,1$. Consider It\^o's formula for $\phi_{j, \eps}(\bfX(t))$ applied component-wise.
\begin{align*}
\mathrm{d} \norm{ \phi_{j,\eps}\big(\bfX(t)\big)}_{H_d}^2 &= 2 \scp{ \phi_{j,\eps} \big(\bfX(t)\big) \phi_{j,\eps}' \big( \bfX(t)\big)}{ \mathbf{f}\big(t, \bfX(t)\big)}_{H_d} \dt\\
&\quad + 2\scp{ \phi_{j,\eps} \big( \bfX(t)\big) \phi_{j,\eps}' \big( \bfX(t)\big)}{ \mathbf{B}\big(t,\bfX(t)\big) \,\mathrm{d} \mathbf{W}(t)}_{H_d}\\
&\quad + \norm{\mathbf{B}\big(t,\bfX(t)\big)^\ast \phi_{j, \eps}'\big(\bfX(t)\big)}_{H_d}^2 \dt\\
&\quad + \scp{ \phi_{j,\eps} \big( \bfX(t)\big) \phi_{j,\eps}'' \big(\bfX(t)\big) }{\norm{ \mathbf{B}\big(t,\bfX(t)\big)}_{L_2(H_d)}^2}_{H_d} \dt.
\end{align*}
In the limit $\eps \to 0$ the stochastic integral vanishes $\PP$-a.\,s. as well as the It\^o correction term since the integrands are nonzero on disjoint sets. Then, $\phi_{j,\eps} \phi_{j,\eps}'' \to 0$ as $\eps \to 0$ lets the latter summand disappear. Thus, by Lebesgue's dominated convergence theorem there only remains the drift part that is
\begin{align*}
\mathrm{d} \norm{ \phi_j\big(\bfX(t)\big)}_{H_d}^2 &= 2 \scp{ \phi_j \big(\bfX(t)\big)}{ \mathbf{f}\big(t, \bfX(t)\big)}_{H_d} \dt\\
&= 2 \sum_{i=1}^d \int_0^1 \phi_j\big(X_i(t)\big) f_i \big(U(t), \bfX(t)\big) \dx \dt \leq 0
\end{align*}
by Assumption \ref{assum:Invariance}. If follows $\norm{ \phi_j\big(\bfX(t)\big)}_{H_d}^2 \leq \norm{ \phi_j\big(\bfx_0 \big)}_{H_d}^2$ $\PP$-a.\,s. Obviously, this implies $\phi_j (\bfX(t)) = 0$ $\PP$-a.\,s. for all $t \in [0,T]$ and a.\,e. $x \in \Dom$ and in conclusion $0 \leq \bfX(t) \leq 1$.
\end{proof}

\subsection{Proof of Theorem \ref{thm:Ex+Unique}}
Define the approximating sequence $(U^n, \bfX^n)$ as follows. Let $U^0 \equiv u_0$ and $\bfX^0 \equiv \bfx_0$. For $n \geq 1$ let $U^n$ be the solution to
\begin{equation}\label{eq:Un}
\mathrm{d} U^n(t) = \Big( A U^n(t) + f \big( U^n(t), \bfX^{n-1}(t)\big) \Big) \dt + B \dwt
\end{equation}
with initial condition $U^n(0) = u_0$. Furthermore, let $\bfX^n$ be the solution to
\begin{equation}\label{eq:Xn}
\mathrm{d} X_i^n(t) = f_i \big(U^n(t), \bfX^n(t)\big) \dt + B_i \big( U^n(t), \bfX^n(t)\big) \dwit, \quad 1 \leq i \leq d,
\end{equation}
with $\bfX^n(0)=\bfx_0$. According to Lemmas \ref{lem:ExforU}--\ref{lem:InvarianceX} these processes exist and are unique. In particular, $U^n \in \mathcal{U}$ and $\bfX^n \in \mathcal{X}$ for all $n \geq 0$. Apparently, we can study the differences $U^{n+1} - U^n$ and $\bfX^{n+1} - \bfX^n$ for $n\geq 1$ in $H$ and $H_d$, respectively. More precisely, it holds that
\begin{align}
\begin{split}\label{eq:DifferenceU}
\mathrm{d} \big(U^{n+1}(t) - U^n(t)\big) &= A  \big(U^{n+1}(t) - U^n(t)\big) \dt\\
&\quad + \Big( f\big( U^{n+1}(t), \bfX^n(t)\big) - f\big(U^n(t), \bfX^{n-1}(t)\big)\Big)\dt,
\end{split}\\
\begin{split}\label{eq:DifferenceX}
\mathrm{d} \big(X_i^{n+1}(t) - X_i^n(t)\big) &= \Big( f_i\big( U^{n+1}(t), \bfX^{n+1}(t)\big) - f_i\big(U^n(t), \bfX^n(t)\big)\Big)\dt\\
&\quad + \Big( B_i\big( U^{n+1}(t), \bfX^{n+1}(t)\big) - B_i\big(U^n(t), \bfX^n(t)\big)\Big)\dwit,
\end{split}
\end{align}
for $1 \leq i \leq d$. Let us start with an useful but elementary estimate due to Assumption \ref{assum:Drift}. For all $u,v \in [-R_t, R_t]$, $\bfx, \bfy \in [0,1]^d$ it holds that
\begin{equation}\label{eq:LocalLipschitz}
\begin{split}
\abs{f(u, \bfx) - f(v, \bfy)} &\leq L \big( 1 + R_t^{r-1}\big) ( 1 + \rho_0) \big( \abs{u-v}^2 + \abs{\bfx - \bfy}^2 \big)^\frac12,\\
\abs{f_i(u, \bfx) - f_i(v, \bfy)} &\leq 2 L \big( 1 + \rho_i (R_t)\big) \big( \abs{u-v}^2 + \abs{\bfx - \bfy}^2 \big)^\frac12.
\end{split}
\end{equation}
At this point it might be noteworthy that the pathwise $L^\infty$-estimate from Lemma \ref{lem:PathwiseBoundforU} is uniform in $n$ since each equation is driven by the same realization of the cylindrical Wiener process $W$. Corresponding to these Lipschitz constants define the process $(G_t)_{t \in [0,T]}$ by
\begin{equation}\label{eq:ProcessG}
G_t \df \int_0^t \Big(2L^2 \big( 1 + R_s^{r-1}\big)^2 (1 + \rho_0)^2 + 4L^2 \sum_{i=1}^d \big( 1 + \rho_i(R_s)\big)^2 + K \big(dL^2 + 1\big) \Big)\ds.
\end{equation}
The next step are differential inequalities for the differences in $H$ and $H_d$. \eqref{eq:DifferenceU} and \eqref{eq:LocalLipschitz} for $f$ and Young's inequality imply
\begin{align*}
&\ddt \norm{U^{n+1}(t) - U^n(t)}_H^2 = 2\dualp{A \big( U^{n+1}(t) - U^n(t)\big)}{U^{n+1}(t) - U^n(t)}\\
&\qquad + 2\dualp{f\big( U^{n+1}(t), \bfX^n(t)\big) - f\big(U^n(t), \bfX^{n-1}(t)\big)}{U^{n+1}(t) - U^n(t)}\\
&\quad \leq - 2\norm{U^{n+1}(t) - U^n(t)}_V^2\\
&\qquad + 2 \norm{f\big( U^{n+1}(t), \bfX^n(t)\big) - f\big(U^n(t), \bfX^{n-1}(t)\big)}_H \norm{U^{n+1}(t) - U^n(t)}_H\\
&\quad \leq - 2\norm{U^{n+1}(t) - U^n(t)}_V^2 + \norm{\bfX^n(t) - \bfX^{n-1}(t)}_{H_d}^2\\
&\qquad + \Big(L^2 \big( 1 + R_t^{r-1}\big)^2 ( 1 + \rho_0)^2 + 1\Big) \norm{U^{n+1}(t) - U^n(t)}_H^2.
\end{align*}
As a consequence, we can obtain the following pathwise estimate.
\begin{equation}\label{eq:FinalDiffUn}
\begin{split}
&\sup_{t \in [0,T]} \mathrm{e}^{-G_t} \norm{U^{n+1}(t) - U^n(t)}_H^2 + 2 \int_0^T \mathrm{e}^{-G_t} \norm{U^{n+1}(t) - U^n(t)}_V^2 \dt\\
&\quad \leq \int_0^T \mathrm{e}^{-G_t} \norm{\bfX^n(t) - \bfX^{n-1}(t)}_{H_d}^2 \dt \quad \PP\text{-a.\,s.}
\end{split}
\end{equation}
In contrast to the case above, the multiplicative noise in $\bfX$ only allows for mean square estimates. However, in a similar fashion one can obtain the analogous inequality for the second variable using \eqref{eq:DifferenceX}, the local Lipschitz conditions for each $f_i$ from \eqref{eq:LocalLipschitz} together the global Lipschitz continuity of each $b_i$.
\begin{align*}
&\mathrm{d} \norm{\bfX^{n+1}(t) - \bfX^n(t)}_{H_d}^2 = 2 \scp{\mathbf{f}\big( U^{n+1}(t), \bfX^{n+1}(t)\big) - \mathbf{f} \big(U^n(t), \bfX^n(t)\big)}{\bfX^{n+1}(t) - \bfX^n(t)}_{H_d} \dt \\
&\qquad+ 2 \scp{ \bfX^{n+1}(t) - \bfX^n(t)}{ \Big( \mathbf{B}\big(U^{n+1}(t), \bfX^{n+1}(t)\big) - \mathbf{B} \big( U^n(t), \bfX^n(t)\big)\Big) \,\mathrm{d}\mathbf{W}(t)}_{H_d}\\
&\qquad+ \norm{\mathbf{B}\big(U^{n+1}(t), \bfX^{n+1}(t)\big) - \mathbf{B} \big( U^n(t), \bfX^n(t)\big)}_{L_2(H_d)}^2 \dt\\
&\quad \leq (2dL^2 + 1) \norm{U^{n+1}(t) - U^n(t)}_H^2 \dt\\
&\qquad + \Big(4 L^2 \sum_{i=1}^d \big(1 + \rho_i(R_t)\big)^2 + 2dL^2 + 1\Big) \norm{\bfX^{n+1}(t) - \bfX^n(t)}_{H_d}^2 \dt\\
&\qquad+ 2 \scp{ \bfX^{n+1}(t) - \bfX^n(t)}{\Big( \mathbf{B}\big(U^{n+1}(t), \bfX^{n+1}(t)\big) - \mathbf{B}\big(U^n(t), \bfX^n(t)\big)\Big) \,\mathrm{d}\mathbf{W}(t)}_{H_d}.\\
\end{align*}
Here, the additional terms are due to the multiplicative noise. Similar to \eqref{eq:FinalDiffUn} we need the exponential of $-G_t$ in the final estimate as follows.
\begin{align*}
&\EV{\int_0^T \Big( G_t - 4L^2 \sum_{i=1}^d \big(1 + \rho_i(R_t)\big)^2 - 2 dL^2 - 1 \Big) \mathrm{e}^{-G_t} \norm{\bfX^{n+1}(t) - \bfX^n(t)}_{H_d}^2 \dt}\\
&+ \EV{\sup_{t \in [0,T]} \mathrm{e}^{-G_t} \norm{\bfX^{n+1}(t) - \bfX^n(t)}_{H_d}^2} \leq \big(2 d L^2 +1\big) \EV{ \int_0^T \mathrm{e}^{-G_t} \norm{U^{n+1}(t) - U^n(t)}_H^2 \dt}\\
&+ \EV{2{\kern -0.5em}\sup_{t \in [0,T]} \abs{ \int_0^t \mathrm{e}^{-G_s} \scp{ \bfX^{n+1}(s) - \bfX^n(s)}{\Big( \mathbf{B}\big(U^{n+1}(s), \bfX^{n+1}(s)\big) - \mathbf{B}\big(U^n(s), \bfX^n(s)\big)\Big) \,\mathrm{d}\mathbf{W}(s)}_{H_d}}}.
\end{align*}
The supremum of the local martingale is controlled via the Burkholder-Davis-Gundy inequality by its quadratic variation. Denote, for the moment, the stochastic integral by $M_t$, then it is straightforward to obtain
\begin{align*}
&\langle M\rangle_T = \int_0^T \mathrm{e}^{-2 G_t} \norm{\bfX^{n+1}(t) - \bfX^n(t)}_{H_d}^2 \norm{ \mathbf{B}\big(U^{n+1}(t), \bfX^{n+1}(t)\big) - \mathbf{B}\big(U^n(t), \bfX^n(t)\big)}_{L_2(H_d)}^2 \dt\\
&\quad \leq 2dL^2 \sup_{t \in [0,T]} \mathrm{e}^{-G_t} \norm{\bfX^{n+1}(t) - \bfX^n(t)}_{H_d}^2\\
&\qquad \times \int_0^T \mathrm{e}^{-G_t}\Big( \norm{U^{n+1}(t) - U^n(t)}_H^2 + \norm{\bfX^{n+1}(t) - \bfX^n(t)}_{H_d}^2 \Big)\dt.
\end{align*}
Young's inequality allows to absorb both factors up to the difference $U^{n+1} - U^n$ by the left hand side and we obtain
\begin{equation}\label{eq:FinalDiffXn}
\EV{ \sup_{t \in [0,T]} \mathrm{e}^{-G_t}  \norm{\bfX^{n+1}(t) - \bfX^n(t)}_{H_d}^2}\leq C \int_0^T \EV{ \mathrm{e}^{-G_t} \norm{U^{n+1}(t) - U^n(t)}_H^2} \dt.
\end{equation}
Thus, \eqref{eq:FinalDiffUn} and \eqref{eq:FinalDiffXn} can be iterated in order to obtain an estimate independent of $n$. In more detail, using Fubini's theorem we can calculate as below
\begin{align*}
&\EV{ \sup_{t \in [0,T]} \mathrm{e}^{-G_t} \norm{U^{n+1}(t) - U^n(t)}_H^2} \leq \int_0^T \EV{\sup_{s \in [0,t]} \mathrm{e}^{-G_s} \norm{\bfX^n(s) - \bfX^{n-1}(s)}_{H_d}^2} \dt\\
&\quad \leq C \int_0^T \int_0^t \EV{ \sup_{r \in [0,s]} \mathrm{e}^{-G_r} \norm{U^n(r) - U^{n-1}(r)}_H^2} \ds \dt\\
&\quad = C \int_0^T (T-s) \EV{ \sup_{r \in [0,s]} \mathrm{e}^{-G_r} \norm{U^n(r) - U^{n-1}(r)}_H^2} \ds.\\
&\quad\leq C^2 \int_0^T \int_r^T (T-s)(s-r) \ds\, \EV{ \sup_{t \in [0,r]} \mathrm{e}^{-G_t} \norm{U^{n-1}(t) - U^{n-2}(t)}_H^2} \dr
\end{align*}
for $n\geq 2$. In general this involves integrals of the form
\begin{equation}\label{eq:IntegralRelation}
\int_r^T (T-s)^\alpha (s-r) \ds = \frac{(T-r)^{\alpha+2}}{(\alpha+1) (\alpha + 2)}.
\end{equation}
With this information the desired inequality is
\begin{equation}\label{eq:FinalRecursionIneq}
\begin{split}
&\EV{ \sup_{t \in [0,T]} \mathrm{e}^{-G_t} \norm{U^{n+1}(t) - U^n(t)}_H^2 + 2 \int_0^T \mathrm{e}^{-G_t} \norm{U^{n+1}(t) - U^n(t)}_V^2 \dt}\\
&\quad \leq \frac{C^n}{(2n-1)!} \int_0^T (T-t)^{2n+1} \EV{ \phantom{\Big|}{\kern -0.2em}\mathrm{e}^{-G_t} \norm{U^1(t) - U^0(t)}_H^2} \dt.
\end{split}
\end{equation}
By \eqref{eq:FinalDiffXn} we get a similar inequality for the differences of $\bfX^n$, thus Borel-Cantelli's lemma yields $\PP$-a.\,s. convergence of $U^n \to U$ on $C([0,T]; H) \cap L^2([0,T];V)$ as well as $\bfX^n \to \bfX$ in $C([0,T]; H_d)$. Also, all $L^\infty$-bounds are uniform in $n$, thus in particular $U \in \mathcal{U}$ and $\bfX \in \mathcal{X}$. This immediately yields an integrable dominating function for all of the integrals in \eqref{eq:Un} and \eqref{eq:Xn} involving $f$ and $f_i$, as well as for the quadratic variation of the multiplicative noise in \eqref{eq:Xn}. Thus, by Lebesgue's dominated convergence theorem $(U, \bfX)$ indeed solves \eqref{eq:SRDE} and the standard a priori estimates follow by Lemmas \ref{lem:ExforU} and \ref{lem:ExforX}.\qed

\section{Finite Difference Approximation}\label{sec:Approx}

In this second part we study spatial approximations of equation \eqref{eq:SRDE} needed for the numerical simulation of the neuronal dynamics using the well known finite difference method. The domain $\Dom$ is approximated by the equidistant grid $\frac1n \{0, \dots, n\}$ and the vectors $U^n, X_i^n \in \R^{n+1}$ denote the functions $U, X_i$ evaluated on this grid. For a compact notation we use bold symbols for matrices $\bfx \in \R^{d \times (n+1)}$ in this section, too. Furthermore, let $\tu$ denote the linear interpolation with respect to the space variable $x$ as
\[
\tu(x) \df (nx - k +1) u_k + (k -nx) u_{k-1},\quad x \in \left[\frac{k-1}{n}, \frac{k}{n}\right]
\]
of the vector $u \in \R^{n+1}$ together with zero outer normal derivative at the boundary points $x=0$ and $x=1$. Also, denote by $\tilde{\bfx}(x)$ the linear interpolation component-wise and by $\iota_n: V_n \to V; u \mapsto \tilde{u}$ the embedding into $V$ (or similarly $V_d$).

At every interior point of the grid we approximate the second derivative by
\begin{equation}\label{eq:DiscreteLaplace}
(A^n v)_k \df n^2 (v_{k+1} - 2 v_k + v_{k-1}),\quad 1 \leq k \leq n-1, v \in \R^{n+1}.
\end{equation}
It is standard to choose centered differences modeling the Neumann boundary condition in order to retain the order of convergence of the interior points. This introduces the artificial variables $v_{-1}, v_{n+1}$ and the discrete boundary condition reads as
\begin{equation}\label{eq:DiscreteBoundary}
\frac{n}{2}(v_1 - v_{-1}) = 0,\quad \frac{n}{2} (v_{n+1} - v_{n-1}) = 0.
\end{equation}
Together with \eqref{eq:DiscreteLaplace} for $k=0, n$ we can eliminate the artificial variables and obtain
\begin{equation}
(A^n v)_0 = 2 n^2 (v_1 - v_0),\quad (A^n v)_n = -2 n^2 (v_n - v_{n-1}).
\end{equation}
Note that $A^n$ is not symmetric with respect to the standard inner product as in the case of a Neumann boundary approximation of order $n^{-1}$. However, introduce the spaces $V_n \cong \R^{n+1}$ with norm
\[
\abs{v}_n^2 \df \frac{1}{2n} \big( v_0^2 + v_n^2 \big) + \frac{1}{n} \sum_{k=1}^{n-1} v_k^2
\]
and corresponding inner product $\scp{\cdot}{\cdot}_n$. Furthermore, we need the semi-norm
\[
\norm{v}_n^2 \df n \sum_{k=1}^n \big(v_k - v_{k-1}\big)^2.
\]
We use the same notation also for the matrices $\bfx$, meaning $\abs{\bfx}_n^2 = \sum_{i=1}^d \abs{\bfx_{i \cdot}}_n^2$ or with $\norm{\cdot}_n$ instead. With respect to $\scp{\cdot}{\cdot}_n$ the matrix $A^n$ is again symmetric and thus the following summation by parts formula holds.
\begin{equation}\label{eq:SummationByParts}
\scp{A^n v}{u}_n = - n \sum_{k=1}^n \big(v_k - v_{k-1}\big)\big(u_k - u_{k-1}\big),\quad \forall u,v \in V_n.
\end{equation}
In the next step let us construct the approximating noise in terms of the given realization of the driving cylindrical Wiener process $W$. Denote by $I_k \df ( \frac{2k-1}{2n}, \frac{2k+1}{2n})$ if $1 \leq k \leq n-1$ and $I_0 \df (0, \frac{1}{2n})$, $I_n \df (\frac{2n-1}{2n}, 1)$. Recall that
\begin{equation}\label{eq:BM}
\scp{W(t)}{\abs{I_k}^{-\frac12} \mathbbm{1}_{I_k}}_H \fd \beta^n_{k}(t), \quad 0 \leq k \leq n
\end{equation}
defines a family of $n+1$ iid real valued Brownian motions. Similarly we define the additional independent $d(n+1)$ Brownian motions $\{ \beta^n_{j,k}\}$ due to $\mathbf{W}$.  The spatial covariance structure given by the kernels $b$ and $b_i$ is discretized as follows,
\begin{align*}
b^n_{k,l} &\df \big(\abs{I_k}\abs{I_l}\big)^{-1} \int_{I_k}\int_{I_l} b(x, y) \dy \dx, \quad 0 \leq k,l \leq n \text{ and}\\
b^n_{i,k,l}(u, \bfx) &\df \big(\abs{I_k}\abs{I_l}\big)^{-1} \int_{I_k}\int_{I_l} b_i(\tu, \tilde{\bfx}, x, y) \dy \dx, \quad 0 \leq k,l \leq n.
\end{align*}
for $u \in V_n, \bfx \in \R^{d \times (n+1)}$. These discrete matrices allows to replace both $b$ and $b_i$ by a piecewise constant kernel, i.\,e. for $ 0\leq k \leq n$ and $x \in I_k$
\begin{align*}
&B W(t) \approx \sum_{l=0}^n b^n_{k,l} \scp{W(t)}{\mathbbm{1}_{I_l}} = \sum_{l=0}^n \abs{I_l}^{-\frac12} b^n_{k,l} \beta_l(t) \fd \Big( B^n P_n W(t)\Big)_k,
\end{align*}
where $P_n u = (\scp{u}{\abs{I_l}^{-\frac12} \mathbbm{1}_{I_l}})_{1 \leq l \leq n}$. In the same way we obtain $B_i^n(u,\bfx) P_n W_i(t)$, $1\leq i \leq d$. Denote by $W^n \df P_n W$ and $W_i^n \df P_n W_i$, $1\leq i\leq d$, the resulting $n$-dimensional Brownian motions and the finite dimensional system of stochastic differential equations approximating equation \eqref{eq:SRDE} is
\begin{equation}\label{eq:Approximation}
\begin{aligned}
\mathrm{d} U^n(t) &= \Big( A^n U^n(t) + f^n \big(U^n(t), \bfX^n(t)\big) \Big) \dt + B^n \, \mathrm{d} W^n(t), & U^n(0) &= P_n u_0,\\
\mathrm{d} X_i^n(t) &=  f_i\big(U^n(t),\bfX^n(t)\big) \dt + B^n_i\big(U^n(t), \bfX^n(t)\big) \, \mathrm{d} W_i^n(t), & X_i^n(0) &= P_n X_i(0).
\end{aligned}
\end{equation}

Standard results on stochastic differential equations imply the existence of a unique strong solution $(U^n(t),\bfX^n(t))$ to \eqref{eq:Approximation}, see e.\,g. \cite[Chapter 3]{prevot}. In order to compute the error made by using the approximate solution we embed $U^n$ and $X^n_i$ into $C([0,T]; V)$ by linear interpolation in the space variable as $(\tU^n, \tilde{\bfX}^n)$. We can now state the main result of this part.
\begin{thm}\label{thm:Approx}
Suppose the assumptions from Theorem \ref{thm:Ex+Unique} are satisfied and recall the definition of $G_t$ in \eqref{eq:ProcessG}. Define the error as $E^n(t) \df ( U(t) - \tU^n(t), \bfX(t) - \tilde{\bfX}^n(t))$, then there exists a constant $C_{\ref{thm:Approx}}$ such that
\begin{equation}\label{eq:NotSoStrongError}
\EV{ \sup_{t \in [0,T]} \mathrm{e}^{-G_t} \norm{E^n(t)}_{H_{d+1}}^2} \leq 2\EV{\norm{E^n(0)}_{H_{d+1}}^2} + \frac{C_{\ref{thm:Approx}}}{n^2}.
\end{equation}
\end{thm}
There is an immediate corollary using Borel-Cantelli's lemma.
\begin{cor}\label{cor:PathwiseRate}
For every $\eps \in (0,1)$ there exists a $\PP$-a.\,s. finite random variable $C_\eps(\omega)$ such that
\[
\sup_{t \in [0,T]} \norm{E^n(t)}_{H_{d+1}}^2 \leq \frac{C_\eps}{n^{1-\eps}} \quad \PP\text{-a.\,s.}
\]
\end{cor}
\begin{rem}
Possible generalizations of this theorem include error estimates in $L^p$ for general $2 \leq p < \infty$ as in e.\,g. \cite[Theorem 3.1]{SauerLattice}. In principle this is only a technical matter and involves more general a priori estimates for $U^n$ and $\bfX^n$ than the ones obtained here, but no new techniques or ideas. Having \eqref{eq:NotSoStrongError} for all finite $p$ would also imply a pathwise convergence rate of almost $\sfrac1n$, see e.\,g. \cite[Lemma 2.1]{KloedenNeuenkirch}.
\end{rem}
Theorem \ref{thm:Approx} does not yield a strong convergence rate, because $\exp[G_t]$ is not necessarily integrable. However, if this is the case we can easily deduce a strong convergence rate of $\sfrac1n$ in $L^p$, $p \leq p^\ast < 2$ (with $p^\ast$ depending on the integrability of $G_t$) by H\"older's inequality. Another case in which one can say more about the strong convergence rate is the one, when the drift is one-sided Lipschitz, or in other words quasi-monotone, see \cite{SauerLattice}. We prove another theorem under the following additional assumptions on $f, f_i$.
\begin{assum}\label{assum:Monotone}
Let $f$ and $\mathbf{f} = (f_i)_{1 \leq i \leq d}$ satisfy
\[
\vek{f(u, \bfx) - f(v, \bfy)}{ \mathbf{f}(u, \bfx) - \mathbf{f}(v, \bfy)}\vek{u - v}{\bfx - \bfy} \leq L \big( \abs{u-v}^2 + \abs{\bfx - \bfy}^2 \big)
\]
for some $L > 0$ and all $u,v \in \R$, $\bfx, \bfy \in \R^d$.
\end{assum}
\begin{thm}\label{thm:Monotone}
With the additional Assumption \ref{assum:Monotone}, there exists a constant $C_{\ref{thm:Monotone}}$ such that
\begin{equation}\label{eq:StrongError}
\EV{ \sup_{t \in [0,T]} \norm{E^n(t)}_{H_{d+1}}^2} \leq 2 \mathrm{e}^{LT}\EV{ \norm{E^n(0)}_{H_{d+1}}^2} + \frac{C_{\ref{thm:Monotone}}}{n^2}.
\end{equation}\end{thm}
The proofs of these theorems are contained in the following subsections.

\subsection{Uniform A Priori Estimates}
In addition to the a priori estimates on $(U, \bfX)$ from Theorem \ref{thm:Ex+Unique} uniform a priori estimates for $(U^n, \bfX^n)$ are essential for the proof of Theorem \ref{thm:Approx}. Let us start with a statement concerning the $L^\infty$-bounds.
\begin{lem}\label{lem:BoundApprox}
Under the assumptions of Theorem \ref{thm:Approx}, $\tilde{\bfX}^n \in \mathcal{X}$ and $\tilde{U}^n \in \mathcal{U}$ with the same uniform bound $R_t$ as for $U$.
\end{lem}
\begin{proof}
Obviously $\tilde{U}^n \in C([0,T], C_b^1(\overline{\Dom}, \R))$. We can apply Lemma \ref{lem:InvarianceX} or rather imitate its proof to obtain $\tilde{\bfX}^n \in \mathcal{X}$ for all $n \in \N$. Since all arguments are point-wise in $x \in \Dom$, they also apply for the finite subset $\frac{1}{n} \{0, \dots, n\}$.

Concerning the uniform bound, note that for the solution to
\[
\mathrm{d}Y^n(t) = A^n Y^n(t) \dt + B^n \,\mathrm{d}W^n(t), \quad Y^n(0)=0,
\]
it holds that $R_t^{Y^n} \df \sup_{s \in [0,t]} \max_{0\leq k\leq n} \abs{Y^n_k(s)} \leq R_t^Y$ by Lemma \ref{app:Lem2}. In particular, the right hand side is independent of $n$ and we can apply the proof of Lemma \ref{lem:PathwiseBoundforU} with the same uniform cut-off.
\end{proof}
Next, we derive an improved a priori estimate providing more than the canonical regularity for $U^n$. This is similar to the question of $U$ being a strong solution to \eqref{eq:SPDEforU}, in particular if $U \in D(A)$.
\begin{lem}\label{lem:ImprovedAPriori}
With the assumptions of Theorem \ref{thm:Approx} and arbitrary $n \in \N$ it holds that
\begin{equation}\label{eq:ImprovedAPriori}
\begin{split}
&\EV{ \sup_{t \in [0,T]} \mathrm{e}^{-G_t} \big( \norm{U^n(t)}_n^2 + \norm{\bfX^n(t)}_n^2 \Big) + 2 \int_0^T \mathrm{e}^{-G_t} \abs{A^n U^n(t)}_n^2 \dt}\\
&\quad\leq K \EV{ \Big( \norm{U^n(0)}_n^2 + \norm{\bfX^n(0)}_n^2 + T \norm{B}_{L_2(H,V)}^2\Big)}.
\end{split} 
\end{equation}
\end{lem}
\begin{proof}
Consider It\^o's formula applied to $\norm{U^n(t)}_n^2$
\begin{align*}
\mathrm{d} \norm{U^n(t)}_n^2 &= 2n \sum_{k=1}^n \big( U^n_k(t) - U^n_{k-1}(t) \big) \Big( \big(A^n U^n(t)\big)_k - \big(A^n U^n(t)\big)_{k-1}\Big) \dt\\
&\quad + 2n \sum_{k=1}^n \big( U^n_k(t) - U^n_{k-1}(t) \big) \Big( f\big(U^n_k(t), \bfX^n_k(t) \big) - f \big(U^n_{k-1}(t), \bfX^n_{k-1}(t) \big) \Big) \dt\\
&\quad + 2 \sqrt{n} \sumkl \big(U^n_k(t) - U^n_{k-1}(t)\big) \big(b_{k,l}^n - b_{k-1,l}^n\big) \,\mathrm{d} \beta_l^n(t) + \sumkl \big( b^n_{k,l} - b^n_{k-1,l}\big)^2 \dt.\\
\intertext{With the summation by parts formula \eqref{eq:SummationByParts} for the linear part and inequality \eqref{eq:LocalLipschitz} for $f$ it is straightforward to obtain}
&\leq \Big( - 2 \abs{A^n U^n(t)}_n^2 + \norm{\bfX^n(t)}_n^2 + \big( L^2 \big(1 + R_t^{r-1}\big)^2 (1 + \rho_0)^2 + 1\big) \norm{U^n(t)}_n^2 \Big.\\
&\qquad \Big. + \norm{B}_{L_2(H,V)}^2 \Big) \dt + 2 \sqrt{n} \sumkl \big(U^n_k(t) - U^n_{k-1}(t)\big) \big(b_{k,l}^n - b_{k-1,l}^n\big) \,\mathrm{d} \beta_l^n(t).
\end{align*}
The norm of $B$ in $L_2(H,V)$ appears naturally as described below in detail for the multiplicative noise in the $\bfX$-variable. The estimate above involves $\bfX^n$ in a stronger norm, thus in a similar fashion apply It\^o's formula to $\norm{X_i^n(t)}_n^2$, $1 \leq i \leq d$.
\begin{align*}
&\mathrm{d} \norm{X^n_i(t)}_n^2 = 2n \sum_{k=1}^n \big( X^n_{i,k}(t) - X^n_{i, k-1}(t)\big) \Big( f_i \big( U^n_k(t), \bfX^n_k(t)\big) - f_i\big( U^n_{k-1}(t), \bfX^n_{k-1}(t)\big)\Big)\dt\\
&\quad + 2 \sqrt{n} \sumkl \big( X^n_{i,k}(t) - X^n_{i, k-1}(t)\big) \Big( b^n_{i,k,l} \big( U^n(t), \bfX^n(t)\big) - b^n_{i,k-1,l}\big( U^n(t), \bfX^n(t)\big)\Big) \,\mathrm{d}\beta_{i,l}^n(t)\\
&\quad + \sumkl \Big( b^n_{i,k,l}\big(U^n(t), \bfX^n(t)\big) - b^n_{i,k-1,l}\big(U^n(t), \bfX^n(t)\big)\Big)^2 \dt.
\end{align*}
Again \eqref{eq:LocalLipschitz} and the linear growth condition on $b_i(\tU^n, \tilde{\bfX}^n) \in W^{1,2}(\Dom^2)$, in particular
\begin{align*}
&\sumkl {\kern 0.25em} \int_{I_k} {\kern -0.4em}\int_{I_{k-1}} {\kern -0.35em}\int_{I_l}{\kern -0.4em} \frac{\big( b_i \big(\tU^n(t,x), \tilde{\bfX}^n(t,x), x, y\big) - b_i\big(\tU^n(t,x'), \tilde{\bfX}^n(t,x'), x', y\big)\big)^2}{\abs{I_k} \abs{I_{k-1}} \abs{I_l}} \dy \dx' \dx\\
&\quad \leq \sumkl {\kern 0.25em}\int_{I_k \cap I_{k-1}} {\kern -0.75em}\Big(\partial_z (b_i(\tU^n(t,z), \tilde{\bfX}^n(t,z), z, y))\Big)^2 \dz \dy \leq \norm{b_i \big(\tU^n(t), \tilde{\bfX}^n(t)\big)}_{W^{1,2}(\Dom^2)}^2
\end{align*}
imply
\begin{align*}
&\mathrm{d} \norm{\bfX^n(t)}_n^2 \leq \Big( \Big( 4L^2 \sum_{i=1}^d \big(1 + \rho_i(R_t)\big)^2 + 2dL^2 + 1\Big) \norm{\bfX^n(t)}_n^2 + (2dL^2 + 1) \norm{U^n(t)}_n^2\Big) \dt\\
& + 2 \sqrt{n} \sum_{i=1}^d\sumkl \big( X^n_{i,k}(t) - X^n_{i, k-1}(t)\big) \Big( b^n_{i,k,l} \big( U^n(t), \bfX^n(t)\big) - b^n_{i,k-1,l}\big( U^n(t), \bfX^n(t)\big)\Big) \,\mathrm{d}\beta_{i,l}^n(t)
\end{align*}
Recall the definition of $G_t$ in \eqref{eq:ProcessG} and combine the inequalities above to
\begin{align*}
&\EV{ \sup_{t \in [0,T]} \mathrm{e}^{-G_t} \big( \norm{U^n(t)}_n^2 + \norm{\bfX^n(t)}_n^2 \Big) + 2 \int_0^T \mathrm{e}^{-G_t} \abs{A^n U^n(t)}_n^2 \dt}\\
&\quad\leq \EV{ \Big( \norm{U^n(0)}_n^2 + \norm{\bfX^n(0)}_n^2 + T \norm{B}_{L_2(H,V)}^2\Big) + \sup_{t \in [0,T]} \abs{M^U(t)} + \sup_{t \in [0,T]} \abs{M^{\bfX}(t)}},
\end{align*}
where the two stochastic integrals are denoted by $M^U$ and $M^{\bfX}$, respectively. The Burkholder-Davis-Gundy inequality and some standard estimates yield
\[
\EV{\sup_{t \in [0,T]} \abs{M^U(t)}} \leq \frac12 \sup_{t \in [0,T]} \mathrm{e}^{-G_t} \norm{U^n(t)}_n^2 + K T \norm{B}_{L_2(H,V)}^2
\]
and similarly
\[
\EV{\sup_{t \in [0,T]} \abs{M^{\bfX}(t)}} \leq \frac12 \sup_{t \in [0,T]} \mathrm{e}^{-G_t} \norm{\bfX^n(t)}_n^2 + K dL^2 \int_0^T \mathrm{e}^{-G_t} \big(\norm{U^n(t)}_n^2 + \norm{\bfX^n(t)}_n^2\big) \dt,
\]
which concludes the proof.
\end{proof}

\subsection{Proof of Theorem \ref{thm:Approx}}
In this part we estimate all parts contributing to the error $E^n(t) \df ( U(t) - \tU^n(t), \bfX(t) - \tilde{\bfX}^n(t))$. Recall the equation for the linear interpolation $\tU^n = \iota_n U^n$ and $\tbfX^n = \iota_n \bfX^n$.
\begin{align*}
\mathrm{d} \tU^n(t) &= \Big( \iota_n A^n U^n(t) + \iota_n f^n \big(U^n(t), \bfX^n(t)\big) \Big) \dt + \iota_n B^n \, \mathrm{d} W^n(t), & \tU^n(0) &= \iota_n P_n u_0,\\
\mathrm{d} \tilde{X}_i^n(t) &=  \iota_n f_i\big(U^n(t),\bfX^n(t)\big) \dt + \iota_n B^n_i\big(U^n(t), \bfX^n(t)\big) \, \mathrm{d} W_i^n(t), & \tilde{X}_i^n(0) &= \iota_n P_n X_i(0).
\end{align*}
In the following we consider everything at some fixed time $t$ and for a clear presentation we drop the explicit time dependence in the notation.

\noindent\textbf{Approximation Error of the Laplacian:} We have that
\begin{align*}
&2\dualp{A U - \iota_n A^n U^n}{U - \tU^n}\\
&\quad= 2\dualp{A U - A \tU^n}{ U - \tU^n} + 2\dualp{A \tU^n - \iota_n A^n U^n}{U - \tU^n}\\
&\quad\leq - 2\norm{U - \tU^n}_V^2 + 2\norm{A \tU^n - \iota_n A^n U^n}_{V^\ast} \norm{U - \tU^n}_V\\
&\quad \leq - \norm{U - \tU^n}_V^2 + \norm{A \tU^n - \iota_n A^n U^n}_{V^\ast}^2
\end{align*}
Furthermore, it holds that
\begin{equation}
A \tU^n = \frac{1}{2n} \Big( \big(A^n U^n\big)_0 \delta_0(x) + \big(A^n U^n\big)_n \delta_1(x)\Big) + \frac{1}{n} \sum_{k=1}^{n-1} \big(A^n U^n)_k \delta_{\frac{k}{n}}(x)
\end{equation}
and
\begin{equation}
\iota_n A^n U^n = \sum_{k=1}^n \big(A^n U^n\big)_k (nx - k+1) \mathbbm{1}_{[ \frac{k-1}{n}, \frac{k}{n}]}(x) + \big(A^n U^n)_{k-1} (k-nx) \mathbbm{1}_{[ \frac{k-1}{n}, \frac{k}{n}]}(x).
\end{equation}
Now let $\phi \in C^\infty(\Dom)$, then
\begin{align*}
&\dualp{A \tU^n - \iota_n A^n U^n}{\phi}\\
&= \sum_{k=1}^{n-1} \big(A^n U^n \big)_k \Big( \frac{1}{n} \phi \big(\tfrac{k}{n}\big) - \int_{\frac{k-1}{n}}^{\frac{k}{n}} (nx-k+1) \phi(x) \dx - \int_{\frac{k}{n}}^{\frac{k+1}{n}} (k + 1 - nx) \phi(x) \dx \Big)\\
&\quad+ \big(A^n U^n\big)_0 \Big( \frac{1}{2n} \phi(0) - \int_0^{\frac{1}{n}} (1 - nx) \phi(x) \dx \Big)\\
&\quad + \big(A^n U^n\big)_n \Big( \frac{1}{2n} \phi(1) - \int_{1-\frac{1}{n}}^1 (nx - n+1) \phi(x) \dx \Big)\\
&= \sum_{k=1}^{n-1} \big(A^n U^n \big)_k \Big( \int_{\frac{k-1}{n}}^{\frac{k}{n}} {\kern -0.5em}(nx-k+1) \big( \phi \big(\tfrac{k}{n}\big) - \phi(x)\big) \dx + \int_{\frac{k}{n}}^{\frac{k+1}{n}} {\kern -0.5em} (k + 1 - nx) \big( \phi \big(\tfrac{k}{n}\big) - \phi(x)\big) \dx \Big)\\
&\quad+ \big(A^n U^n\big)_0 \int_0^{\frac{1}{n}} (1 - nx)\big( \phi(0) - \phi(x)\big) \dx + \big(A^n U^n\big)_n \int_{1-\frac{1}{n}}^1 (nx - n+1) \big(\phi(1) - \phi(x) \big) \dx.
\end{align*}
Obviously, we can write $\phi(\frac{k}{n}) - \phi(x) = \int_x^{\frac{k}{n}} \phi' (y) \dy$ and together with Cauchy-Schwarz' and Jensen's inequality it follows that
\begin{align*}
\dualp{A \tU^n - \iota_n A^n U^n}{\phi} \leq \frac{2}{n} \abs{A^n U^n}_n \Bigg( \sum_{k=1}^n \int_{\frac{k-1}{n}}^{\frac{k}{n}} \phi'(y)^2 \dy \Bigg)^\frac12 = \frac{2}{n} \abs{A^n U^n}_n \norm{\phi}_V.
\end{align*}
This inequality extends to all $\phi \in V$ and in particular we obtain that
\begin{equation}
2\dualp{AU - \iota_n A^n U^n}{U - \tU^n} \leq - \norm{U - \tU^n}_V^2 + \frac{4}{n^2} \abs{A^n U^n}_n^2.
\end{equation}
\textbf{Approximation Error of the Nonlinear Drift:} Let us start with the error coming from the first variable. There we obtain
\begin{align*}
&2 \scp{f\big(U, \bfX\big) - \iota_n f \big(U^n, \bfX^n\big)}{ U - \tU^n}_H\\
&\quad\leq 2\norm{ f\big(U, \bfX\big) - f \big( \tU^n, \tbfX^n\big)}_H \norm{ U - \tU^n}_H + 2\scp{f \big( \tU^n, \tbfX^n\big) - \iota_n f \big(U^n, \bfX^n\big)}{U - \tU^n}_H\\
&\quad\leq \Big(L^2 \big( 1 + R_t^{r-1} \big)^2 ( 1 + \rho_0)^2 + 1 \Big) \norm{U - \tU^n}_H^2 + \norm{\bfX - \tbfX^n}_{H_d}^2 \\
&\qquad+ 2\scp{f \big( \tU^n, \tbfX^n\big) - \iota_n f \big(U^n, \bfX^n\big)}{U - \tU^n}_H
\end{align*}
by \eqref{eq:LocalLipschitz}. The latter term can be estimated as follows
\begin{align*}
&2\scp{f \big( \tU^n, \tbfX^n\big) - \iota_n f \big(U^n, \bfX^n\big)}{U - \tU^n}_H\\
&\quad= \sumk \intwithlimits{k} \Big[ (nx - k+1) \Big( f\big(\tU^n(x), \tbfX^n(x)\big) - f\big(U^n_k, \bfX^n_k\big)\Big)\Big.\\
&\qquad\qquad\qquad \Big.+ (k - nx)\Big( f\big(\tU^n(x), \tbfX^n(x)\big) - f\big( U^n_{k-1}, \bfX^n_{k-1}\big) \Big)\Big] \big( U(x) - \tU^n(x)\big) \dx.\\
\intertext{With \eqref{eq:LocalLipschitz}, the relations $\abs{\tU^n(x) - U^n_k} = (k-nx) \abs{U^n_k - U^n_{k-1}}$, as well as $\abs{\tU^n(x) - U^n_{k-1}} = (nx-k+1) \abs{U^n_k - U^n_{k-1}}$ and of course the same ones for $\bfX^n$ it follows that}
&\quad\leq 4 L \big( 1 + R_t^{r-1} \big) ( 1 + \rho_0) \sumk \big( \abs{U^n_k - U^n_{k-1}}^2 + \abs{\bfX^n_k - \bfX^n_{k-1}}^2\big)^\frac12\\
&\qquad \times \intwithlimits{k} (nx-k+1) (k-nx) \abs{U(x) - \tU^n(x)} \dx\\
&\quad \leq L^2 \big( 1 + R_t^{r-1} \big)^2 ( 1 + \rho_0)^2 \norm{U-\tU^n}_H^2 + \frac{K}{n^2} \Big( \norm{U^n}_n^2 + \norm{\bfX^n}_n^2 \Big).
\end{align*}
Thus, in the end we have obtained the following estimate.
\begin{equation}\label{eq:ErrorF}
\begin{split}
&2 \scp{f\big(U, \bfX\big) - \iota_n f \big(U^n, \bfX^n\big)}{ U - \tU^n}_H \leq  \frac{K}{n^2} \Big( \norm{U^n}_n^2 + \norm{\bfX^n}_n^2 \Big)\\
&\qquad +\Big(2L^2 \big( 1 + R_t^{r-1} \big)^2 ( 1 + \rho_0)^2 + 1 \Big) \norm{U-\tU^n}_H^2 + \norm{\bfX - \tbfX^n}_{H_d}^2.
\end{split}
\end{equation}
The estimates for each $f_i$ work the same way, in particular it holds that
\begin{equation}\label{eq:ErrorFi}
\begin{split}
&2 \scp{\mathbf{f}\big(U, \bfX\big) - \iota_n \mathbf{f} \big(U^n, \bfX^n\big)}{ \bfX - \tbfX^n}_{H_d} \leq  \frac{K}{n^2} \Big( \norm{U^n}_n^2 + \norm{\bfX^n}_n^2 \Big)\\
&\qquad + \Big(4L^2 \sum_{i=1}^d \big( 1 + \rho_i(R_t) \big)^2 + 1 \Big) \norm{\bfX-\tbfX^n}_{H_d}^2 + \norm{U - \tU^n}_H^2.
\end{split}
\end{equation}
\textbf{Approximation Error of the Covariance Operators:} 
\begin{align*}
&\norm{B - \iota_n B^n P_n}_{L_2(H)}^2\\
&\quad= {\kern -0.5em} \sum_{k=1, l=0}^n {\kern -0.25em} \intwithlimits{k} \intwithlimits{l} \Big[ (nx-k+1) \big( b(x,y) - b^n_{k,l}\big) + (k-nx) \big( b(x,y) - b^n_{k-1,l} \big) \Big]^2 \dy \dx
\end{align*}
Observe that for $x \in I_k$, $y \in I_l$
\begin{align*}
b(x,y) - b^n_{k,l} &= \big( \abs{I_k} \abs{I_l}\big)^{-1} \int_{I_k} \int_{I_l} b(x,y) - b(x', y) + b(x', y) - b(x', y') \dy' \dx'\\
&= \big( \abs{I_k} \abs{I_l}\big)^{-1} \int_{I_k} \int_{I_l} \int_{x'}^x \partial_z b(z,y) \dz + \int_{y'}^y \partial_z b(x', z) \dz \dy' \dx'.
\end{align*}
With this relation it is straightforward to obtain
\begin{equation}\label{eq:ErrorB}
\begin{split}
\norm{B - \iota_n B^n P_n}_{L_2(H)}^2 &\leq \frac{2}{n^2} \int_0^1 \int_0^1 \Big( \abs{\partial_x b(x,y)}^2 + \abs{\partial_y b(x,y)}^2 \Big) \dy \dx\\
&\leq \frac{2}{n^2} \norm{b}_{W^{1,2}(\Dom^2)}^2 = \frac{2}{n^2} \norm{B}_{L_2(H,V)}^2.
\end{split}
\end{equation}
Of course, the It\^o correction coming from the $\bfX$-variable can be estimated similarly, with a combination of the calculation above and the one for the nonlinear drift.
\begin{align*}
&\norm{B_i(U, \bfX) - \iota_n B_i^n (U^n, \bfX^n) P_n}_{L_2(H)}^2\\
&\quad= 2 \norm{B_i(U, \bfX) - B_i(\tU^n, \tbfX^n)}_{L_2(H)}^2 + 2 \norm{B_i(\tU^n, \tbfX^n) - \iota_n B_i^n (U^n, \bfX^n) P_n}_{L_2(H)}^2\\
&\quad\leq 4 L^2 \Big( \norm{U- \tU^n}_H^2 + \norm{\bfX - \tbfX^n}_{H_d}^2\Big) + 2 \norm{B_i(\tU^n, \tbfX^n) - \iota_n B_i^n (U^n, \bfX^n) P_n}_{L_2(H)}^2
\end{align*}
and the latter term is
\[
\norm{B_i(\tU^n, \tbfX^n) - \iota_n B_i^n (U^n, \bfX^n) P_n}_{L_2(H)}^2 \leq \frac{K}{n^2} \norm{B_i\big( \tU^n, \tilde{\bfX}^n\big)}_{L_2(H,V)}^2 \leq \frac{K L^2}{n^2} \big( \norm{U^n}_n^2 + \norm{\bfX^n}_n^2\big).
\]
\begin{proof}[Proof of Theorem \ref{thm:Approx}]
We can apply It\^o's formula to the square of the $H_{d+1}$-norm of $E^n(t)$ and obtain
\begin{align*}
\mathrm{d} \norm{E^n(t)}_{H_{d+1}}^2 &= 2 \dualp{A U(t) - \iota_n A^n U^n(t)}{U(t) - \tU^n(t)} \dt\\
&\quad + 2 \scp{f\big( U(t), \bfX(t)\big) - \iota_n f \big( U^n(t), \bfX^n(t)\big)}{U(t) - \tU^n(t)}_H \dt\\
&\quad + 2 \scp{ \mathbf{f} \big(U(t), \bfX(t)\big) - \iota_n \mathbf{f} \big( U^n(t), \bfX^n(t)\big)}{ \bfX(t) - \tbfX^n(t)}_{H_d} \dt\\
&\quad + 2 \scp{U(t) - \tU^n(t)}{\Big( B - \iota_n B^n P_n\Big) \dwt}_H + \norm{ B - \iota_n B^n P_n}_{L_2(H)}^2 \dt\\
&\quad + 2 \scp{\bfX(t) - \tbfX^n(t)}{ \Big( \mathbf{B} \big( U(t), \bfX(t)\big) - \iota_n \mathbf{B}^n \big(U^n(t), \bfX^n(t)\big) P_n \Big) \, \mathrm{d} \mathbf{W}(t)}_{H_d}\\
&\quad + \sum_{i=1}^d \norm{ B_i\big( U(t), \bfX(t)\big) - \iota_n B_i^n \big( U^n(t), \bfX^n(t)\big) P_n}_{L_2(H)}^2 \dt.
\end{align*}
Everything above has been estimated in the three steps before, except for the stochastic integrals. Recall $G_t$ and we need to estimate the corresponding equation with $\exp[-G_t]$ in the supremum over $t \in [0,T]$. Obviously, this involves the Burkholder-Davis-Gundy inequality after integration with respect to $\PP$. The standard decomposition, see e.\,g. Lemma \ref{lem:ImprovedAPriori}, bounds the quadratic variation of the stochastic integral from above in terms of the left hand side and the already investigated It\^o correction term yields the following inequality in a straightforward way
\begin{align*}
\EV{ \sup_{t \in [0,T]} \mathrm{e}^{-G_t} \norm{ E^n(t)}_{H_{d+1}}^2} &\leq 2 \EV{ \norm{ E^n(0)}_{H_{d+1}}^2} + \frac{8}{n^2} \EV{ \int_0^T \mathrm{e}^{-G_t} \abs{A^n U^n(t)}_n^2 \dt}\\
&\quad + \frac{K}{n^2} \big(dL^2 + 1 \big) \EV{ \int_0^T \mathrm{e}^{-G_t} \big( \norm{U^n(t)}_n^2 + \norm{\bfX^n(t)}_n^2 \big) \dt}.
\end{align*}
By Lemma \ref{lem:ImprovedAPriori} we know that the expectations on the right hand side are uniformly bounded in $n$, hence there exists $C_{\ref{thm:Approx}}$ with 
\[
\EV{ \sup_{t \in [0,T]} \mathrm{e}^{-G_t} \norm{ E^n(t)}_{H_{d+1}}^2} \leq 2\EV{ \norm{ E^n(0)}_{H_{d+1}}^2} + \frac{C_{\ref{thm:Approx}}}{n^2}.\qedhere
\]
\end{proof}

\subsection{Proof of Theorem \ref{thm:Monotone}}
Apparently the problematic term $\exp[-G_t]$ is due to the nonlinear drift. Going back to the derivation of the error estimates \eqref{eq:ErrorF} and \eqref{eq:ErrorFi} we observe that with a slightly different application of Young's inequality and the one-sided Lipschitz condition follows
\begin{equation}
\begin{split}
&2\langle \vek{f(U, \bfX) - \iota_n f(U^n, \bfX^n)}{\mathbf{f}(U, \bfX) - \iota_n \mathbf{f}(U^n, \bfX^n)}, \vek{U-\tU^n}{\bfX - \tilde{\bfX}^n}\rangle_{H_{d+1}}\\
&\quad\leq 2L \big( \norm{U - \tU^n}_H^2 + \norm{\bfX - \tilde{\bfX}^n}_{H_d}^2 \big) + \frac{K}{n^2} G_t \big( \norm{U^n}_n^2 + \norm{\bfX^n}_n^2 \big).
\end{split} 
\end{equation}
Gronwall's inequality now implies
\begin{align*}
\EV{ \sup_{t \in [0,T]} \norm{ E^n(t)}_{H_{d+1}}^2} &\leq \mathrm{e}^{4L T}\EV{ \norm{ E^n(0)}_{H_{d+1}}^2} + \frac{4}{n^2} \mathrm{e}^{4L T} \EV{ \int_0^T \abs{A^n U^n(t)}_n^2 \dt}\\
&\quad + \frac{K}{n^2} \big(dL^2 + 1 \big)\mathrm{e}^{4L T} \EV{ \int_0^T G_t \big( \norm{U^n(t)}_n^2 + \norm{\bfX^n(t)}_n^2 \big) \dt}.
\end{align*}
It remains to show that the right hand side is finite, basically a modification of Lemma \ref{lem:ImprovedAPriori}, which yields a priori estimates in $L^p$ for $p>2$. Then, with H\"older's inequality and the Gaussian moments of $G_t$ we can immediately conclude Theorem \ref{thm:Monotone}. 
\begin{cor}\label{cor:ImprovedAPriori}
For all $n \in \N$ and $1\leq p < \infty$ it holds that
\begin{gather*}
\EV{\int_0^T \abs{A^n U^n(t)}_n^2 \dt} \leq \mathrm{e}^{(L + dL^2) T} \EV{ \norm{U^n(0)}_n^2 + \norm{\bfX^n(0)}_n^2} + T \norm{B}_{L_2(H,V)}^2,\\
\EV{ \sup_{t \in [0,T]} \big( \norm{U^n(t)}_n^2 + \norm{\bfX^n(t)}_n^2 \Big)^p} \leq C \EV{ \Big( \norm{U^n(0)}_n^{2p} + \norm{\bfX^n(0)}_n^{2p} + T \norm{B}_{L_2(H,V)}^{2p}\Big)}.
\end{gather*}
\end{cor}
Clearly, this concludes the assertion using H\"older's inequality, since $G_t$ has finite moments of any order. The proof of this a priori estimate is essentially the same as the one of Lemma \ref{lem:ImprovedAPriori} with the observation, that we only need
\[
\vek{f(U^n_k, \bfX^n_k) - f(U^n_{k-1}, \bfX^n_{k-1})}{ \mathbf{f}(U^n_k, \bfX^n_k) - \mathbf{f}(U^n_{k-1}, \bfX^n_{k-1})}\vek{U^n_k - U^n_{k-1}}{\bfX^n_k - \bfX^n_{k-1}} \leq L \big( \abs{U^n_k - U^n_{k-1}}^2 + \abs{\bfX^n_k - \bfX^n_{k-1}}^2 \big)
\]
and $G_t$ does not appear at all, in particular for all $t \in [0,T]$ it holds that
\begin{align*}
\norm{U^n(t)}_n^2 + \norm{\bfX^n(t)}_n^2 + 2 \int_0^t \abs{A^n U^n(s)}_n^2 \ds &\leq \big( 2L + 2dL^2\big) \int_0^t \big( \norm{U^n(s)}_n^2 + \norm{\bfX^n(s)}_n^2 \big) \ds\\
&\quad + \norm{B}_{L_2(H,V)}^2 t + M^U(t) + M^{\bfX}(t),
\end{align*}
where again the stochastic integrals are denoted by $M^u$ and $M^{\bfX}$, respectively. Integration with respect to $\PP$ and Gronwall's inequality directly yields the first a priori estimate. For the second one consider the inequality above to the power $1< p < \infty$.
\begin{align*}
K \sup_{t \in [0,T]} \big(\norm{U^n(t)}_n^2 + \norm{\bfX^n(t)}_n^2\big)^p &\leq \big( 2L + 2dL^2\big)^p T^{p-1} \int_0^T \big( \norm{U^n(t)}_n^2 + \norm{\bfX^n(t)}_n^2 \big)^p \dt\\
&\quad + \norm{B}_{L_2(H,V)}^{2p} T^p + \sup_{t \in [0,T]} \abs{M^U(t)}^p + \sup_{t \in [0,T]} \abs{M^{\bfX}(t)}^p.
\end{align*}
The Burkholder-Davis-Gundy inequality and Gronwall's lemma yield the result. \qed 

\section{Applications}
As an application for our main results we consider two equations describing the propagation of the action potential in a neuron along the axon. More precisely we study the Hodgkin-Huxley equations and the FitzHugh-Nagumo equations mostly popular in the mathematical literature. In particular, one can use Theorem \ref{thm:Monotone} to extend the results of \cite{SauerLattice}.
\subsection{Stochastic Hodgkin-Huxley Equations}\label{sec:HH}
As already described in the introduction, the Hodgkin-Huxley equations, see \cite{HodgkinHuxley}, are the basis for all subsequent conductance based models for active nerve cells. In the case of the squid's giant axon, the currents generating the action potential are primarily due to sodium and potassium ions and the membrane potential satisfies
\begin{equation}\label{eq:HHmembrane}
\tau \partial_t U = \lambda^2 \partial_{xx} U - g_{\text{Na}} ( U - E_{\text{Na}}) - g_{\text{K}} (U - E_{\text{K}}) - g_{\text{L}} (U - E_{\text{L}}).
\end{equation}
The latter term is the leak current with $g_{\text{L}} > 0$, $E_{\text{Na}}, E_{\text{K}} \in \R$ are the resting potentials of sodium and potassium and $g_{\text{Na}}, g_{\text{K}}$ are their conductances. $\tau$ and $\lambda$ are the specific time and space constants of the axon. Due to opening and closing of ion channels the conductances $g_{\text{Na}}, g_{\text{K}}$ may change with time, in particular
\[
g_{\text{Na}} = \overline{g}_{\text{Na}} m^3 h \quad \text{and} \quad g_{\text{K}} = \overline{g}_{\text{K}} n^4, \quad \overline{g}_{\text{Na}}, \overline{g}_{\text{K}} > 0,
\]
where $(n,m,h)$ are gating variables describing the probability of ion channels being open. Let $x = n,m,h$, then
\begin{equation}\label{eq:HHgating}
\frac{\mathrm{d} x}{\dt} = \alpha_x(U) ( 1 - x) - \beta_x(U) x
\end{equation}
with typical shapes
\[
\alpha_x(U) = a_x^1 \frac{U + A_x}{1 - \mathrm{e}^{-a_x^2 (U + A_x)}} \geq 0 \quad \text{and} \quad \beta_x(U) = b_x^1 \mathrm{e}^{-b_x^2 (U + B_x)}\geq 0
\]
for some constants $a_x^i, b_x^i > 0$, $A_x, B_x \in \R$. For the data matching the ``standard Hodgkin-Huxley neuron'' we refer to e.\,g. \cite[Section 1.9]{Ermentrout}.

Let $\bfx = (n,m,h)$ and we implicitly defined the nonlinearities $f, f_i$ above. It is now immediate to see that Assumption \ref{assum:Drift} is satisfied with $r=2$,
\[
\rho(\bfx) \df \max \{ \abs{n}^4, \abs{m}^3 \abs{h}, \abs{m}^3, \abs{m}^2 \abs{h}, \abs{n}^3\},
\]
and
\[
\rho_i(u) \df  \max \{ \alpha_{x_i}(u) + \beta_{x_i}(u), \alpha_{x_i}'(u) + \beta_{x_i}'(u)\}.
\]
Note that the Lipschitz constants do indeed grow exponentially. Concerning Assumption \ref{assum:Invariance} we observe that for $\bfx \in [0,1]^3$ it holds that
\[
\partial_u f(u, \bfx) = - \overline{g}_{\text{Na}} m^3 h - \overline{g}_{\text{K}} n^4 -  g_{\text{L}} \leq - g_{\text{L}} < 0,
\]
hence $K = 0$. Also, the invariance of $[0,1]^3$ follows by
\[
f_i(u, x_i) \mathbbm{1}_{\{x_i \leq 0\}} = \Big(\alpha_{x_i} (u) - \big( \alpha_{x_i}(u) + \beta_{x_i}(u)\big) x_i \Big) \mathbbm{1}_{\{x_i \leq 0\}} \geq \alpha_{x_i} (u)\mathbbm{1}_{\{x_i \leq 0\}} \geq 0
\]
and
\[
f_i(u, x_i) \mathbbm{1}_{\{x_i \geq 1 \}} = \Big(\big( \alpha_{x_i}(u) + \beta_{x_i}(u)\big) (1 - x_i) - \beta_{x_i}(u) \Big) \mathbbm{1}_{\{x_i \geq 1\}} \leq - \beta_{x_i} (u)\mathbbm{1}_{\{x_i \geq 1\}} \leq 0.
\]
So far we only stated the deterministic system, however adding noise in both variables is physiologically reasonable, see \cite{GoldwynSheaBrown} and we can study the following system.
\begin{align*}
\tau \mathrm{d}U(t) &= \Big(\lambda^2 A U(t) - \overline{g}_{\text{Na}} m(t)^3 h(t) \big( U(t) - E_{\text{Na}}\big) \Big.\\
&\qquad \Big. - \overline{g}_{\text{K}} n(t)^4 \big(U(t) - E_{\text{K}}\big) - g_{\text{L}} \big(U(t) - E_{\text{L}}\big) \Big) \dt + B \dwt,\\
\mathrm{d}n(t) &= \Big( \alpha_n \big(U(t)\big) \big( 1 - n(t)\big) - \beta_n\big(U(t)\big) n(t)\Big) \dt\\
&\qquad + \mathbbm{1}_{\{0 \leq n(t) \leq 1\}} \sigma_n n(t) \big( 1 - n(t)\big) B_n \,\mathrm{d}W_n(t),\\
\mathrm{d}m(t) &= \Big( \alpha_m \big(U(t)\big) \big( 1 - m(t)\big) - \beta_m\big(U(t)\big) m(t)\Big) \dt\\
&\qquad + \mathbbm{1}_{\{0 \leq m(t) \leq 1\}} \sigma_m m(t) \big( 1 - m(t)\big) B_m \,\mathrm{d}W_m(t),\\
\mathrm{d}h(t) &= \Big( \alpha_h \big(U(t)\big) \big( 1 - h(t)\big) - \beta_h\big(U(t)\big) h(t)\Big) \dt\\
&\qquad + \mathbbm{1}_{\{0 \leq h(t) \leq 1\}} \sigma_h h(t) \big( 1 - h(t)\big) B_h \,\mathrm{d}W_h(t).
\end{align*}
Here, $W, W_n, W_m$ and $W_h$ are cylindrical Wiener processes on $H= L^2(\Dom)$ and $B, B_n, B_m$ and $B_h$ are Hilbert-Schmidt operators $\in L_2(H,V)$, hence both Assumption \ref{assum:Drift} and \ref{assum:Invariance} are satisfied and we can apply Theorem \ref{thm:Ex+Unique} and \ref{thm:Approx} for the stochastic version of the system \eqref{eq:HHmembrane} + \eqref{eq:HHgating}.
\subsection{Stochastic FitzHugh-Nagumo Equations}
As in \cite{SauerLattice}, we consider the spatially extended stochastic FitzHugh-Nagumo equations as a second example. These equations were originally stated by FitzHugh in \cite{FitzHugh1}, \cite{FitzHugh2} as a system of ODEs simplifying the Hodgkin-Huxley model in terms of only two variables $U$ and $\bfX = w$, the so-called recovery variable. See e.\,g. the monograph \cite{Ermentrout} for more details on the deterministic case. With the original parameters, the equations are
\begin{equation}\label{eq:FHN}
\begin{split}
\mathrm{d}U(t) &= \Big( A U(t) + \big( U(t) - \tfrac13 U(t)^3\big) - w(t) \Big) \dt + B \dwt,\\
\mathrm{d}w(t) &= 0.08 \big( U(t) - 0.8 w(t) + 0.7\big) \dt,
\end{split}
\end{equation}
where $W$ is a cylindrical Wiener process on $H = L^2(\Dom)$ and $B \in L_2(H,V)$. One can easily check that Assumption \ref{assum:Drift}, \ref{assum:Kernel} and \ref{assum:Monotone} are satisfied, however not the second part of Assumption \ref{assum:Invariance}, which is not surprisingly since $w$ is not representing a proportion anymore and thus $[0,1]$ is not forward invariant for the dynamics. On the other hand, we have
\[
\abs{\nabla f(u,w)} \leq 1 + \abs{u}^2 \quad \text{and}\quad \abs{\nabla f_w(u,w)} \leq L,
\]
hence $w$ does not appear in the Lipschitz constants and in particular
\[
G_t \df \int_0^t \Big( 4\big(1 + R_s^4\big) + K \big(L^2 + 1\big)\Big) \ds
\]
is independent of any $L^\infty$-bound for $w$ and therefore the invariance of $[0,1]$ for $w$ is obsolete. With these modifications we can use Theorem \ref{thm:Monotone} to obtain a strong rate of convergence of $\sfrac1n$, improving \cite[Theorem 3.1]{SauerLattice}.

\appendix
\section{Appendix}
The remaining part provides some straightforward estimates concerning the stochastic convolution and its approximation used throughout this article. Here, we do not aim for the most generality or even optimality of the results since this is a different matter.

Let $E \df C( \overline{\Dom}, \R)$ and $Y$ the unique (mild) solution to
\begin{equation}\label{app:OU}
\mathrm{d} Y(t) = A Y(t) \dt + B \dwt, \quad Y(0) = 0.
\end{equation}
\begin{lem}\label{app:Lem1}
Define by $\xi \df \norm{Y}_{C([0,T]; E)} < \infty$ $\PP$-a.\,s. a random variable on $\Omega$ that satisfies $\EV{ \exp [\alpha \xi^2]} < \infty$ for some $\alpha > 0$. In particular, $\xi$ has finite moments of any order.
\end{lem}
\begin{proof}
Since $A$ generates an exponentially stable, analytic semigroup $\{\mathrm{e}^{tA}\}_{t \geq 0}$ on $E$ we can do the following integration by parts to obtain a different representation of $Y$ as the mild solution to \eqref{app:OU}.
\[
Y(t) = \int_0^t \mathrm{e}^{(t-s)A} B \dws = \int_0^t A \mathrm{e}^{(t-s)A} B \big( W(s) - W(t)\big) \ds + \mathrm{e}^{tA} B W(t).
\]
In particular, this does not involve a stochastic integral anymore, thus standard estimates for $A$ yield
\begin{align*}
\norm{Y(t)}_E &\leq \int_0^t \norm{A \mathrm{e}^{(t-s)A} B \big( W(s) - W(t) \big)}_E \ds + \norm{ \mathrm{e}^{tA} BW(t)}_E\\
&\leq C_A \int_0^t (t-s)^{-1} \norm{BW(s) - BW(t)}_E \ds + C_A \norm{BW(t)}_E.
\end{align*}
Given $0 < \eta < \sfrac{1}{2}$ the process $\{BW\}_{t \geq 0}$ is $\eta$-H\"older continuous in $E$ by the Sobolev embedding $V \hookrightarrow E$ in $d=1$. Hence, define $\zeta \df \norm{BW}_{C^\eta([0,T];E)}$, which is a Gaussian random variable, see \cite[Theorem 5.1]{VeraarHytoenen}. The integrability then follows from \cite[Corollary 3.2]{LedouxTalagrand}.
\end{proof}
The second lemma concerns a simple uniform estimate for the solution to the approximating problem
\begin{equation}\label{app:OUApprox}
\mathrm{d} Y^n(t) = A^n Y^n(t) \dt + B^n \,\mathrm{d}W^n(t),\quad Y^n(0) = 0.
\end{equation}
\begin{lem}\label{app:Lem2}
For all $n \in \N$ it holds that
\[
\xi^n \df \norm{Y^n}_{C([0,T]; (R^{n+1}, \norm{\cdot}_{\max}))} = \sup_{t \in [0,T]} \max_{0 \leq k \leq n} \abs{ Y^n_k(t)} \leq C_A \zeta
\]
with the Gaussian random variable $\zeta$ from Lemma \ref{app:Lem1}. In particular, this bound is uniform in $n$ and has the same moments as $\xi$.
\end{lem}
\begin{proof}
As in the proof of Lemma \ref{app:Lem1} we can conclude that $\xi^n \leq C_A \zeta^n$, where $\zeta^n \df \norm{B^n W^n}_{C^\eta([0,T]; (\R^{n+1}, \norm{\cdot}_{\max}))}$ for some $0 < \eta < \sfrac12$. We write
\[
\big( B^n W^n(t)\big)_k = \sum_{l=0}^n b^n_{k,l} \scp{W(t)}{ \mathbbm{1}_{I_l}}_H = n \int_{I_k} \scp{\pi^n b(x, \cdot)}{ W(t)}_H \dx,
\]
where $\pi^n u \df \sum_{l=0}^n \scp{u}{\abs{I_l}^{-\sfrac12} \mathbbm{1}_{I_l}}_H \abs{I_l}^{-\sfrac12} \mathbbm{1}_{I_l}$ is the projection onto these finitely many orthonormal indicator functions. It follows for $t \neq s$
\begin{align*}
&\abs{\big(B^n W^n(t) \big)_k - \big(B^n W^n(s)\big)_k} \abs{t-s}^{-\eta} \leq \abs{I_k}^{-1} \int_{I_k} \abs{ \scp{\pi^n b(x,\cdot}{W(t)-W(s)}_H} \dx \abs{t-s}^{-\eta}\\
&\quad \leq \norm{\pi^n}_{L(H)} \norm{BW(t) - BW(s)}_E \abs{t-s}^{-\eta} \leq \norm{BW}_{C^\eta([0,T]; E)} = \zeta.
\end{align*}
Since the right hand side is independent of $k$ as well as $t$ and $s$ the assertion follows.
\end{proof}

\section*{Acknowledgment}
This work is supported by the BMBF, FKZ 01GQ1001B.

\end{document}